\documentclass[final]{siamltex}
\newcommand{\norm}[1]{\left\lVert#1\right\rVert}
\usepackage{amsfonts, amssymb,amsmath}

\newtheorem{remark}[theorem]{Remark}
\usepackage{mathtools}
\usepackage[pdftex]{graphicx}
\usepackage[left=2.5cm,top=3cm,right=2.5cm,bottom=3cm,bindingoffset=0.5cm]{geometry}
\usepackage{float}
\usepackage{chngcntr}
\counterwithout{figure}{section}
\counterwithout{table}{section}

\usepackage{mathdots}

\title{Conditioning and backward error of block-symmetric block-tridiagonal linearizations of matrix polynomials}

\author{M. I. Bueno\thanks{Department of Mathematics and College of Creative Studies,
University of California, Santa Barbara, CA 93106, USA ({\tt mbueno@math.ucsb.edu}). The research of M. I. Bueno was partially
supported by NSF grant DMS-1358884 and partially supported by ``Ministerio de
Econom\'{i}a, Industria y Competitividad of Spain" and ``Fondo Europeo de
Desarrollo Regional (FEDER) of EU" through grants MTM-2015-68805-REDT and
MTM-2015-65798-P (MINECO/FEDER, UE)}.
\and F. M. Dopico \thanks{Departamento de Matem\'{a}ticas, Universidad Carlos III
de Madrid, Avda.\ Universidad 30, 28911 Legan\'{e}s, Spain
(\texttt{dopico@math.uc3m.es}). The research of F. M. Dopico was partially
supported by ``Ministerio de Econom\'{i}a, Industria y Competitividad of Spain"
and ``Fondo Europeo de Desarrollo Regional (FEDER) of EU" through grants
MTM-2015-68805-REDT and MTM-2015-65798-P (MINECO/FEDER, UE).}        \and S. Furtado \thanks{Centro de An\'{a}lise Funcional, Estruturas Lineares e Aplica\c{c}\u{o}es da Universidade de Lisboa and Faculdade de Economia do Porto, Rua Dr. Roberto Frias
4200-464 Porto, Portugal ({\tt sbf@fep.up.pt}).} The research of S.\ Furtado was partially supported by project UID/MAT/04721/2013. \and L. Medina \thanks{Boston University. One Silber Way. Boston, MA 02215, USA  ({\tt medinal@bu.edu}). The research of L. Medina was partially supported by NSF grant DMS-1358884.}}

\begin{document}

\maketitle

\begin{abstract}
For each square matrix polynomial $P(\lambda)$ of odd degree, a  block-symmetric block-tridiagonal pencil  $\mathcal{T}_{P}(\lambda)$, in  the family of generalized Fiedler pencils, was introduced by
Antoniou and Vologiannidis in 2004, and a variation $\mathcal{R}_P(\lambda)$ of this pencil  was introduced by Mackey et al. in 2010.  These two pencils have several appealing properties, namely they are always strong linearizations of $P(\lambda)$, they are easy to construct from the coefficients of $P(\lambda)$,  the eigenvectors of $P(\lambda)$ can be recovered easily from those of $\mathcal{T}_P(\lambda)$ and $\mathcal{R}_P(\lambda)$, the two pencils are symmetric (resp. Hermitian) when  $P(\lambda)$ is,  and they preserve the (classical) sign characteristic of $P(\lambda)$ when $P(\lambda)$ is Hermitian.  In this paper we study the numerical behavior
of $\mathcal{T}_{P}(\lambda)$ and $\mathcal{R}_P(\lambda)$. We compare the conditioning of a finite,
nonzero, simple eigenvalue $\delta$ of $P(\lambda)$, when considered an
eigenvalue of $P(\lambda)$ and an eigenvalue of $\mathcal{T}_{P}(\lambda)$. We
also compare the backward error of an approximate eigenpair $(z,\delta)$ of
$\mathcal{T}_{P}(\lambda)$ with the backward error of an approximate eigenpair
$(x,\delta)$ of $P(\lambda)$, where $x$ was recovered from $z$ in an appropriate way. We  show that analogous results are obtained for $\mathcal{R}_P(\lambda)$. When the matrix coefficients of $P(\lambda)$ have similar norms and $P(\lambda)$ is scaled so that the largest norm of the matrix coefficients of $P(\lambda)$ is one, we conclude that $\mathcal{T}_{P}(\lambda)$  and $\mathcal{R}_P(\lambda)$ have good
numerical properties in terms of eigenvalue conditioning and backward error. 
Moreover, we compare the numerical behavior of $\mathcal{T}%
_{P}(\lambda)$ (and $\mathcal{R}_P(\lambda)$) with that of other well-studied linearizations in the literature,
namely, the first companion linearization $C_{1}(\lambda)$, and the first and last
linearizations in the standard basis of the vector space $\mathbb{DL}(P)$,
introduced by Mackey et al. (2006), and conclude that
$\mathcal{T}_{P}(\lambda)$ performs better than these linearizations when $P(\lambda)$ has odd degree and  has been scaled.  The theoretical results obtained in the paper are illustrated by numerical experiments, which confirm the advantage of using  $\mathcal{T}_P(\lambda)$ (and $\mathcal{R}_P(\lambda)$) over the other linearizations considered in the paper, in particular in symmetric and Hermitian problems. 
\end{abstract}



\pagestyle{myheadings}
\thispagestyle{plain}
\markboth{M.I. Bueno, F.M. Dopico, S. Furtado, and L. Medina}{Conditioning and backward error of block-symmetric linearizations}

\section{Introduction}
This paper focus on the study of numerical properties of some specific  linearizations of a matrix polynomial. We start by introducing some well-known concepts  that will be fundamental in the understanding of the problem that we address.

Let $P(\lambda)$ be an $n\times n$ matrix polynomial of grade $k$ with complex coefficients, i.e., 
\begin{equation}
P(\lambda)=\sum_{i=0}^{k}A_{i}\lambda^{i}, \quad A_{i}\in\mathbb{C}^{n\times n}. \label{pol}%
\end{equation}
If $A_k \neq 0$, we say that $P(\lambda)$ has degree $k$. We say that $P(\lambda)$ is a \emph{regular matrix polynomial}  if  \textrm{det}$(P(\lambda))$ is not identically zero. Otherwise, $P(\lambda)$ is said to be \emph{singular}. In this paper we will focus on regular matrix polynomials. 

 A matrix pencil  $
L(\lambda)= \lambda L_1 - L_0$,
 with $L_{1},L_{0}\in
\mathbb{C}^{kn\times kn}$, is a \emph{linearization} of $P(\lambda)$  of grade $k$ (see
\cite{GLR-book, GLR-book2}) if there exist two unimodular matrix polynomials (i.e. matrix
polynomials with constant nonzero determinant), $U(\lambda)$ and $V(\lambda)$,
such that
\[
U(\lambda)L(\lambda)V(\lambda)=\left[
\begin{array}
[c]{cc}%
I_{(k-1)n} & 0\\
0 & P(\lambda)
\end{array}
\right]  .
\]
Here and hereafter $I_{m}$ denotes the $m\times m$ identity matrix. Clearly, a linearization of a regular matrix polynomial is also regular. 

We denote by $revP(\lambda)$ the matrix polynomial obtained from
$P(\lambda)$ by reversing the order of the matrix coefficients, that is, if $P(\lambda)$ is as in (\ref{pol}), then 
$$rev P(\lambda):=\lambda^k P\left(\frac{1}{\lambda}\right)= \sum_{i=0}^k A_{k-i} \lambda^i.$$
 The matrix polynomial $rev P(\lambda)$ is known as the \emph{reversal} of $P(\lambda)$. When
$L(\lambda)$ is a linearization of $P(\lambda)$ and, in addition,
$revL(\lambda)$ is  a linearization of  $revP(\lambda),$ we say that
$L(\lambda)$ is a \emph{strong linearization} of $P(\lambda).$

Suppose that $P(\lambda)$ is regular. If  $x$ and $y$ are
nonzero vectors such that $P(\delta)x=0$ and $y^{\ast}P(\delta)=0$
for some complex number $\delta$, then $x$ and $y$ are said to be,
respectively, a \emph{right and a left eigenvector} of $P(\lambda)$
associated with the \emph{eigenvalue} $\delta$. 
 The polynomial eigenvalue problem consists in finding the eigenvalues and (right and left)
eigenvectors of a given matrix polynomial  $P(\lambda)$ and appears in many applications \cite{betcke2013nlevp, kailath, mackey, 4m-vspace, goodvibrations, voss, 21}, which has motivated an intense research activity on its numerical solution in the last fifteen years \cite{backward, tisseur, 20,scaling}. The standard way to solve this problem numerically is to construct a linearization $L(\lambda)$ of $P(\lambda)$ and solve
the corresponding generalized eigenproblem $L(\lambda)z=0$ using well-known algorithms, like the QZ algorithm for moderate size problems \cite{moler}, or a projection method for large sparse problems \cite{dongarra}. It is clear from the definition of linearization  that $P(\lambda)$ and 
$L(\lambda)$ share  the finite elementary divisors and, thus, the finite eigenvalues. If $L(\lambda)$ is a strong linearization of $P(\lambda)$, $P(\lambda)$ and $L(\lambda)$ also share the
infinite elementary divisors (that is, $rev P(\lambda)$ and $rev L(\lambda)$ share the elementary divisors associated with the eigenvalue 0). However, (strong) linearizations of $P(\lambda)$ do not share the eigenvectors of $P(\lambda)$. Thus,
for a linearization of $P(\lambda)$ to be useful in solving the polynomial eigenvalue problem for $P(\lambda)$,  among other aspects, 
it is important that the eigenvectors of $P(\lambda)$ can be recovered from the 
eigenvectors of  $L(\lambda)$ in an easy way. Additionally,  if $P(\lambda)$ is a structured matrix polynomial (symmetric, Hermitian, palindromic, etc.), it is important that the 
linearization $L(\lambda)$ shares the same structure so that it preserves numerically, i.e., in the presence of rounding errors,  the  spectral properties imposed by this structure. In this paper, we will be particularly interested in Hermitian matrix polynomials, which arise very often in applications to model systems. A set of signs, called the sign characteristic, can be associated with the eigenvalues of such type of matrix polynomials. This set of signs  is crucial for determining the behavior of the systems described by them, namely, by helping to understand the difference in behavior of the eigenvalues of their eigenvalues under structured and unstructured perturbations. Thus, when solving a polynomial eigenvalue problem using a linearization, it is important to know how the sign characteristic of the linearization and of the matrix polynomial are related \cite{ammari,SCDLP,SC,SC-new}. In particular, it may be convenient to use Hermitian linearizations of a Hermitian matrix polynomial that preserve the sign characteristic of $P(\lambda)$.

 When calculating the eigenvalues and eigenvectors of a linearization of a matrix polynomial $P(\lambda)$, rounding  errors occur. Thus, it is
important to know how these errors affect an approximate eigenpair of $P(\lambda)$ obtained from an approximate eigenpair of the linearization. In this context,  it is desirable that the condition number of an eigenvalue of a linearization $L(\lambda)$ of $P(\lambda)$ is similar to the 
condition number of the same eigenvalue of $P(\lambda)$ and, since the algorithms to solve the generalized eigenvalue problem produce small backward errors on the linearization, that the backward error of an eigenpair of  $P(\lambda)$ is close 
to that of $L(\lambda)$.  The goal of this paper is to study for the first time the numerical properties of the block-tridiagonal block-symmetric linearizations  in the family of generalized Fiedler pencils (GFP) \cite{ant-vol04} in terms of conditioning of eigenvalues and backward errors of approximate eigenpairs, and to compare their behavior with the one of other well-known linearizations in the literature.   We will consider the classical relative normwise condition number of a
simple, finite, nonzero eigenvalue of $P(\lambda)$ and the classical normwise backward error of an 
approximate eigenpair of $P(\lambda)$ (see Section \ref{cone-sec} and \cite{20}; see also Section \ref{cond-D1Dk}). We note that  it is not our objective here   to provide new linearizations nor to study known linearizations from an algebraic point of view. 

In \cite{backward, tisseur}, the relative condition number  of eigenvalues and the backward error of  approximate eigenpairs of the companion linearizations  $C_1(\lambda)$ and $C_2(\lambda)$ of a regular matrix polynomial $P(\lambda)$,  and  of  the linearizations of $P(\lambda)$ in the vector space $\mathbb{DL}(P)$, were  studied  (for a description of $C_1(\lambda)$ and $C_2(\lambda)$, see \cite{GLR-book}, and for a description of the pencils in $\mathbb{DL}(P)$, see \cite{HMMT, 4m-vspace}). 

One of the conclusions in \cite{tisseur} is that, given a simple, finite, nonzero eigenvalue $\delta$ of $P(\lambda)$  as in (\ref{pol}),  if $A_0$ is nonsingular and $|\delta| \geq 1$, then  $D_1(\lambda, P)$, the first pencil in the standard basis of $\mathbb{DL}(P)$, has a condition number close to
optimal among the linearizations of $P(\lambda)$ in $\mathbb{DL}(P)$, while if $A_k$ is nonsingular and $|\delta| \leq 1$, $D_k(\lambda, P)$, the last pencil in the standard basis of $\mathbb{DL}(P)$, has the same property (assuming that the  matrix coefficients of 
$P(\lambda)$  have similar norms).  Moreover, each of these condition numbers is close to that of $P(\lambda)$ for the same
eigenvalue. Recall that $D_1(\lambda, P)$ (resp. $D_k(\lambda, P)$) is a linearization of a regular $P(\lambda)$ if and only if
$A_0$ (resp. $A_k$) is nonsingular (\cite[Theorem 5.5]{BDFM} and \cite[Theorem 6.7]{4m-vspace}). Similarly, in \cite{backward}, it was shown that the 
linearizations $D_1(\lambda, P)$ and $D_k(\lambda, P)$ have optimal properties with respect to backward errors when $|\delta| \geq 1$ and $|\delta| \leq 1$, respectively, assuming that the matrix coefficients of $P(\lambda)$ have similar norms.   
We observe that, when $P(\lambda)$ is Hermitian, $D_1(\lambda,P)$ and $D_k(\lambda, P)$ are also Hermitian and  have attractive properties regarding the sign characteristic. In fact, in \cite{SC} we showed that, if $P(\lambda)$ as in (\ref{pol})  is  Hermitian and $A_k$ is nonsingular, then $D_k(\lambda, P)$ is a strong linearization of $P(\lambda)$ that  preserves the  sign characteristic of $P(\lambda)$ associated with the real eigenvalues. If $k$ is odd, then $D_1(\lambda, P)$ preserves the sign characteristic of $P(\lambda)$ as well.

Regarding the Frobenius companion forms $C_1(\lambda)$ and $C_2(\lambda)$, it is shown in  \cite{tisseur} that  their eigenvalues are potentially more ill conditioned than the same eigenvalues of $P(\lambda)$, though if the spectral norms of the matrix coefficients of $P(\lambda)$ are
approximately 1,  the eigenvalues of $P(\lambda)$ and of the companion linearizations have similar conditions numbers. An analogous behavior  holds with respect to backward errors of approximate eigenpairs of the companion linearizations, when using an appropriate algorithm to recover the
eigenpairs of $P(\lambda)$ from those of the linearizations (see \cite{backward}). 

In this paper, we focus on two block-symmetric block-tridiagonal  pencils associated with a matrix polynomial $P(\lambda)$ of odd degree, which we denote by $\mathcal{T}_P(\lambda)$ and $\mathcal{R}_P(\lambda)$. These pencils   were introduced in \cite{ant-vol04, 4m-alt}.  When
$P(\lambda)$ (regular or singular) has odd degree, these pencils are strong linearizations of  $P(\lambda)$  (with no nonsingularity restrictions on the coefficients of $P(\lambda)$). 
 The pencil $\mathcal{T}_P(\lambda)$  is in the family GFP  while $\mathcal{R}_P(\lambda)$
   is obtained from $\mathcal{T}_P(\lambda)$  by performing some operations on the  block-rows and block-columns, namely
  permutations and multiplications by -1.  These pencils have many relevant properties that may make them very useful in practice. They are companion forms, that is, for $P(\lambda)$ as in (\ref{pol}), they are always strong linearizations and 
their matrix coefficients  are block matrices whose blocks are of the form 
$0$, $\pm I_n$ or $\pm A_i$. Since they are
block-symmetric pencils, they are symmetric (Hermitian) when $P(\lambda)$ is.   The eigenvectors of
$P(\lambda)$ can be easily recovered from the eigenvectors of these two linearizations (see Section \ref{GFP}). Moreover,  they preserve the  sign characteristic  associated with the  real eigenvalues of $P(\lambda)$, when $P(\lambda)$ is 
Hermitian with nonsingular leading coefficient \cite[Theorem 5.3]{SC}.

We will show that, in terms of conditioning and backward errors, $\mathcal{T}_p(\lambda)$ and $\mathcal{R}_p(\lambda)$ also have attractive 
properties when $P(\lambda)$ has odd degree. More precisely, for  a simple, finite, nonzero eigenvalue $\delta$ of a regular matrix polynomial $P(\lambda)$ of odd degree $k$,  we study the condition 
number of $\delta$ as an eigenvalue of $\mathcal{T}_P(\lambda)$ and of $\mathcal{R}_P(\lambda)$. We show that,  when  it is possible to scale  $P(\lambda)$  so that all  its matrix coefficients  have norm approximately equal to 1, the condition number of $\delta$ as an eigenvalue of   $\mathcal{T}_P(\lambda)$ and  of $\mathcal{R}_P(\lambda)$  is comparable to the condition number of $\delta$ as an  eigenvalue of $P(\lambda)$ regardless of the modulus of  $\delta$, which is in stark contrast with the behavior of the linearizations $D_1(\lambda, P)$ and $D_k(\lambda, P)$ described above.  Similarly, we study the relationship between the backward error of an approximate eigenpair of $\mathcal{T}_P(\lambda)$ and of $\mathcal{R}_P(\lambda)$  with the backward error of an approximate eigenpair of $P(\lambda)$ associated with the same eigenvalue and show that both backward errors are comparable when the approximate eigenvector in the eigenpair of $P(\lambda)$ is recovered from the approximate eigenvector of the linearization in a convenient way. Due to all their attractive properties, the pencils $\mathcal{T}_P(\lambda)$ and $\mathcal{R}_P(\lambda)$ can be considered one of the most useful linearizations of $P(\lambda)$ known in the literature, when $P(\lambda)$ has odd degree, and, even more, when $P(\lambda)$ is symmetric or Hermitian.

In fact, the pencils $D_1(\lambda,P)$ and $D_k(\lambda,P)$ have certain disadvantages when used to compute eigenvalues and eigenvectors of a symmetric (Hermitian) regular matrix polynomial $P(\lambda)$.  First of all, when $A_0$ (resp. $A_k$) is close to be singular, it is unlikely that $D_1(\lambda, P)$ (resp. $D_k(\lambda, P)$) will exhibit a good numerical behavior since  $D_1(\lambda, P)$ (resp. $D_k(\lambda, P)$) is not a linearization of $P(\lambda)$ when $A_0$ (resp. $A_k)$ is singular.  Additionally, the optimality of their condition number and backward error depends on the modulus of the eigenvalue that needs to be computed which forces the use of the two linearizations when the polynomial has eigenvalues with modulus less than 1 and eigenvalues with modulus larger than 1. These problems can be solved if, instead of using $D_1(\lambda,P)$ and/or $D_k(\lambda,P)$ as linearizations of $P(\lambda)$, we use $C_1(\lambda)$, which is a strong linearization of   $P(\lambda)$ (regular and singular). However, the main problem with using  $C_1(\lambda)$ is that it is not symmetric (Hermitian) when $P(\lambda)$ is. In contrast with these linearizations,  when $P(\lambda)$ has odd degree, the pencil $\mathcal{T}_P(\lambda)$ is symmetric  (Hermitian) when $P(\lambda)$ is and it is always a strong linearization of   $P(\lambda)$ (regular or singular), i.e. without requiring any conditions on the matrix coefficients of the polynomial. Moreover,
as will be shown,  it  has a numerical behavior similar to $C_1(\lambda)$.


This paper is organized as follows. In Section \ref{cone-sec} we introduce the concept of
condition number  $\kappa_P(\delta)$ of a simple, finite, nonzero eigenvalue $\delta$ of a regular
matrix polynomial $P(\lambda)$ and present the explicit formula for it. We also recall the definition of backward error $\eta_P(x, \delta)$ of an eigenpair $(x, \delta)$ of a matrix polynomial $P(\lambda)$ and provide the explicit formula for it given in \cite{20}. In Section \ref{sec-aux} we introduce some auxiliary concepts and
results that will be helpful in proving the main results of the paper. In Section \ref{GFP} we  present  a
block-symmetric block-tridiagonal GFP  associated with a  matrix polynomial $P(\lambda)$ of odd degree, denoted by $\mathcal{T}_P(\lambda)$, and    a block-symmetric block-tridiagonal pencil strictly equivalent to $\mathcal{T}_P(\lambda)$, denoted $\mathcal{R}_P(\lambda)$.  In Section \ref{GFP} we also explain  how to construct the eigenvectors of these linearizations of $P(\lambda)$ from  eigenvectors of $P(\lambda)$ associated with the same eigenvalues  and viceversa.  
Section \ref{main} is the main section of the paper. We  study the conditioning of eigenvalues and backward error of approximate eigenpairs of the linearizations $\mathcal{T}_P(\lambda)$ and $\mathcal{R}_P(\lambda)$. More precisely, we provide bounds on the ratio $\frac{\kappa_L(\delta)}{\kappa_P(\delta)}$, where $\delta$ is a simple, nonzero, finite eigenvalue of $P(\lambda)$, and $L(\lambda)$ denotes any of the linearizations $\mathcal{T}_P(\lambda)$ or $\mathcal{R}_P(\lambda)$. Upper bounds on the ratio $\frac{\eta_P(x, \delta)}{\eta_L(z, \delta)}$ are also presented, where $(z, \delta)$ is an approximate eigenpair of $L(\lambda)$ and $(x, \delta)$ is an approximate eigenpair of $P(\lambda)$, with $x$ recovered from $z$ in a  convenient way to minimize the upper bounds.  In Section \ref{cond-D1Dk} we recall the definition of the pencils $D_1(\lambda,P)$, $D_k(\lambda,P)$ and $C_1(\lambda)$ and improve some results already known in the literature regarding the comparison of the conditioning of eigenvalues and backward error of eigenpairs of these pencils with those of $P(\lambda)$, when they are linearizations of $P(\lambda)$.  Finally, in Section \ref{numerical}, we present  some numerical experiments that illustrate the theoretical results presented in previous sections and compare the performance of the different linearizations considered in the paper, emphasizing the attractive behavior of $\mathcal{T}_P(\lambda)$ and $\mathcal{R}_P(\lambda)$. 

\section{Condition number and backward error of  matrix polynomials}\label{cone-sec}

In this section we recall the concepts of relative condition number of an eigenvalue and of backward error of an approximate eigenpair of a regular matrix polynomial  of degree $k$  as in (\ref{pol}). 

We note that, in this paper, we only consider simple eigenvalues since these are essentially the only ones appearing in numerical practice. The reason is that they are by far the most common and in the rare occasions where exact multiple eigenvalues are present, they become almost always simple clustered eigenvalues by the effect of rounding errors.





Given a complex vector $x$, we denote by $\|x\|_2$ the Euclidean norm of $x$. For $A \in \mathbb{C}^{n\times n}$, we denote by $\|A\|_2$ the spectral norm of $A$,
that is, the matrix norm of $A$ induced by the Euclidean norm.

Let $\delta$ be a simple, finite, nonzero eigenvalue of a regular matrix polynomial
$P(\lambda)$ of degree $k$ as in (\ref{pol}), and let $x$  be a right 
eigenvector of $P(\lambda)$ associated with $\delta$. A (relative) normwise condition
number $\kappa_P(\delta)$ of $\delta$ can be defined by

\begin{align*}
\kappa_{P}(\delta) = \lim_{\epsilon\to0} \sup\left\{  \frac{|\Delta\delta
|}{\epsilon|\delta|} : [P(\delta+ \Delta\delta) + \Delta P \right.   &
(\delta+ \Delta\delta)](x + \Delta x) =0,\\
&  \left.  \|\Delta A_{i} \|_{2} \leq\epsilon\; \omega_{i}, i=0,\ldots, k \right\}  ,
\end{align*}
where $\Delta P(\lambda)= \sum_{i=0}^{k} \lambda^{i} \Delta A_{i}$ and
$\omega_{i}$, $i=0,\ldots,k$,  are nonnegative weights that allow flexibility in how the
perturbations are measured. This condition number is an immediate generalization of the well-known Wilkinson condition number for the standard eigenvalue problem and measures the relative change in an eigenvalue.

The next theorem gives an explicit formula for this condition number. For a matrix polynomial $P(\lambda)$, we denote by $P'(\lambda)$ the first derivative of $P(\lambda)$ with respect to $\lambda$.


\begin{theorem}\cite[Theorem 5]{20}\label{condform} Let $P(\lambda)$ be a regular matrix polynomial of degree $k$. Let $\delta$ be a simple, finite, nonzero eigenvalue of 
$P(\lambda)$, and let $x$ and $y$ be a right
and a left eigenvector of $P(\lambda)$ associated with $\delta$. The normwise condition
number $\kappa_{P}(\delta)$ is given by
\begin{equation}\label{num-cond}
\kappa_{P}(\delta)=\frac{(\sum_{i=0}^{k}|\delta|^{i}\omega_{i})\Vert
y\Vert_{2}\Vert x\Vert_{2}}{|\delta||y^{\ast}P^{\prime}(\delta)x|}.
\end{equation}
\end{theorem}

Note that, since $\delta$ is nonzero and  simple, the denominator of the expression for  $\kappa_P(\delta)$ given in (\ref{num-cond}) is nonzero \cite[Theorem 3.2]{simpledelta}.

In calculating $\kappa_P(\delta)$, we will use
the weights $\omega_i=\|A_i\|_2$. We will call these weights \emph{the natural weights for $P(\lambda)$}.   In particular,  for a  pencil $\lambda L_1 - L_0$, the natural weights are  $\omega_1=\|L_1\|_2$ and $\omega_0=\|L_0\|_2$.

Note that, if $\delta \neq 0$ is an eigenvalue of $P(\lambda)$, then the (left and right) eigenvectors of $P(\lambda)$ and $rev P(\lambda)=\lambda^k P\left (\frac{1}{\lambda}\right)$ associated with $\delta$ and $\frac{1}{\delta}$, respectively, coincide. Thus, when considering the natural weights for $P(\lambda)$ and $rev P(\lambda)$, we obtain the following result, which is a simple consequence of Theorem \ref{condform}.  

Note that, if $P(\lambda)$ as in (\ref{pol}) has degree $k$ and $A_0\neq 0$, then $rev P(\lambda)$ has degree $k$ as well.
\begin{lemma}
\label{lcondrev}Let $P(\lambda)$ be a regular matrix polynomial of degree $k$ as in (\ref{pol}) with $A_0\neq 0$.
Let $\delta$ be a simple, finite, nonzero eigenvalue of $P(\lambda)$. 
Then
\[
\kappa_{P}(\delta)=\kappa_{revP}\left(  \frac{1}{\delta}\right)  .
\]

\end{lemma}


 In order to find a strong linearization $L(\lambda)$ of $P(\lambda)$ such that $\delta$ has a comparable condition number when considered as an eigenvalue of $P(\lambda)$ and of $L(\lambda)$,  it may  be convenient to scale $P(\lambda)$. 
Let us consider  the eigenvalue parameter scaling  given by $\widetilde{P}(\mu):=\beta P(\gamma
\mu),$ where $\beta$,$\gamma$ are nonzero complex scaling parameters. We have
\begin{equation}\label{Ptilde}
\widetilde{P}(\mu)= \mu^{k}\widetilde{A}_{k}+\cdots
+\mu\widetilde{A}_{1}+\widetilde{A}_{0},
\end{equation}
where $\widetilde{A}_{i}=\beta\gamma^{i}A_{i}$, $i=1,\ldots,k.$   It is clear
that $\delta$ is an eigenvalue of $P(\lambda)$ if and only if $\delta/\gamma$ is an eigenvalue of $\widetilde{P}(\mu)$. Moreover, $P(\lambda)$ and $\widetilde{P}(\mu)$ have the
same eigenvectors for the eigenvalues $\delta$ and $\delta/\gamma$,
respectively. Thus, by considering the natural weights for $P(\lambda)$ and
$\widetilde{P}(\mu)$, it follows that $\kappa_{P}(\delta)$ is invariant under a scaling of the type described above, that is, $\kappa_{P}(\delta)=\kappa
_{\widetilde{P}}(\delta/\gamma).$  We will  show  that  this type of scaling can be used on $P(\lambda)$  to improve   the bounds  of the ratio of the condition numbers of an eigenvalue of $P(\lambda)$ and of  a linearization.  More explicitly, when the norms of the matrix coefficients of $P(\lambda)$ do not vary  too much, we will use  the scaling given by 
 $\gamma=1$ and $\beta^{-1}=\max_{i=0,\ldots, k}\{\|A_i\|_2\}$ to improve such bounds. Otherwise, we will  consider scalings with $\gamma \neq 1$. 

A disadvantage of the condition number $\kappa_{P}(\delta)$ defined above is
that it is not valid for zero or infinite eigenvalues. In \cite[Theorem 4.2]{dedieu},  Dedieu and Tisseur defined a relative eigenvalue condition number available  for all 
eigenvalues, including 0 and infinity. In this definition, the matrix
polynomial $P(\lambda)$ is rewritten in homogeneous form, that is,
\[
P(\alpha,\beta)=\sum_{i=0}^{k}\alpha^{i}\beta^{k-i}A_{i},
\]
and an eigenvalue $\delta$ of $P(\lambda)$ is identified  with any pair $(\alpha, \beta)\neq (0,0)$ for which $\delta = \alpha/ \beta$. We observe that, contrarily to what happens with the condition number in (\ref{num-cond}), which is invariant under some 
scalings of the matrix polynomial (a procedure that we apply to improve the condition number  of linearizations), the condition number of Dedieu
and Tisseur is scale-dependent, that is,  it can be changed by a scaling of the  
 original matrix polynomial  \cite{scaling}. Since we anticipate that the ratio of the Dedieu and Tisseur condition number of a zero or infinite eigenvalue, when considered as an eigenvalue of a matrix polynomial $P(\lambda)$ and of each of the linearizations considered in this paper, may not be good  without a previous  scaling of $P(\lambda)$,   we do not consider such eigenvalues in our study and focus on the condition number in (\ref{num-cond}).  In \cite{scaling} the homogeneous formulation was avoided for similar reasons.
 
The normwise backward error of an approximate (right) eigenpair $(x, \delta)$ of $P(\lambda)$, where $\delta$ is finite, is defined by 

$$\eta_P(x, \delta)=\min\{ \epsilon: (P(\delta)+ \Delta P(\delta))x=0, \quad \|\Delta A_i\|_2 \leq \epsilon \|A_i\|_2, \; i=0,\ldots, k\},$$
where $\Delta P(\lambda)=\sum_{i=0}^k \lambda^i \Delta A_i.$
Similarly, for an approximate left eigenpair $(y^*, \delta)$, we have
$$\eta_P(y^*, \delta):=\min \{ \epsilon: y^* ( P(\delta)+\Delta P(\delta))=0, \quad \|\Delta A_i\|_2 \leq \epsilon \|A_i\|_2, \; i=0,\ldots, k\}.$$

The following result provides  explicit formulas for $\eta_P(x, \delta)$ and $\eta_P(y^*, \delta)$.

\begin{theorem}\label{back}\cite[Theorem 1]{20} Let $P(\lambda)$ be a regular matrix polynomial of degree $k$ as in (\ref{pol}). For a given approximate right eigenpair $(x, \delta)$ of 
$P(\lambda)$, where $x\in \mathbb{C}^{n\times 1}$ and $\delta\in \mathbb{C}$, the normwise backward error $\eta_P(x, \delta)$ is given by
$$\eta_P(x,\delta)=\frac{\|P(\delta)x\|_2}{(\sum_{i=0}^{k} |\delta|^{i}\|A_i\|_2)\|x\|_2}.$$ 
For an approximate left eigenpair $(y^*, \delta)$, where $y\in \mathbb{C}^{n\times 1}$ and $\delta \in \mathbb{C}$, we have 
$$\eta_P(y^*,\delta)=\frac{\|y^* P(\delta)\|_2}{(\sum_{i=0}^{k} |\delta|^{i}\|A_i\|_2)\|y\|_2}.$$ 
\end{theorem}

The  result we present next is a consequence of Theorem \ref{back} and can be easily checked.

\begin{lemma}
\label{lbackrev}Let $P(\lambda)$ be a regular matrix polynomial of degree $k$ as in (\ref{pol})  with $A_0 \neq 0$. For
a given approximate right  (resp. left) eigenpair $(x,\delta)$ (resp. ($y^{\ast},\delta)$)
of $P(\lambda)$, where $x,y\in\mathbb{C}^{n\times1}$ and $0\neq\delta
\in\mathbb{C}$, we have that $(x, \frac{1}{\delta})$ (resp. $(y^*, \frac{1}{\delta}))$ is an approximate right (resp. left) eigenpair of $rev P(\lambda)$ and 
\[
\eta_{P}(x,\delta)=\eta_{revP}\left(  x,\frac{1}{\delta}\right) \quad  \text{ and
}\quad \eta_{P}(y^{\ast},\delta)=\eta_{revP}\left (y^{\ast},\frac{1}{\delta}\right).
\]

\end{lemma}

A scaling on the eigenvalue parameter can also be used to improve the normwise backward error of an approximate eigenpair of a linearization of a matrix polynomial.  It is easy to show that $\eta_{\tilde{P}}(x, \mu)= \eta_{P}(x, \gamma \mu)$, where $\tilde{P}(\mu)$ is as in (\ref{Ptilde}), while the corresponding backward errors of the (same) linearizations of $P(\lambda)$ and 
$\tilde{P}(\mu)$ can be quite different.  We will use this fact in our numerical experiments to show that a scaling of the polynomial can be applied to decrease the backward error of the block-symmetric  linearizations that we are studying in this paper. 


\begin{remark}\label{A0neq0}
To compare the conditioning of eigenvalues and backward error of approximate eigenpairs of a matrix polynomial $P(\lambda)$ with those of the linearizations considered in this paper, it will be convenient to assume that  $P(\lambda)$ as in (\ref{pol}) has a nonzero constant term $A_0$. However, if $A_0=0$, we have that  $P(\lambda)=\lambda^s P_1(\lambda)$  for some $s$, where the constant term of $P_1(\lambda)$ is nonzero, and we may consider $P_1(\lambda)$ instead of $P(\lambda)$ to compute  the nonzero eigenvalues of $P(\lambda)$. Note that $P(\lambda)$ and $P_1(\lambda)$ have the same nonzero finite eigenvalues.
\end{remark}

The  sign characteristic of a Hermitian matrix polynomial  can be obtained from 
the sign characteristic of one of its Hermitian linearizations, as long as we know how they are related. In \cite{SCDLP, SC},  the authors 
considered the classical definition of sign characteristic of a Hermitian matrix polynomial $P(\lambda)$  with nonsingular
leading coefficient, presented in \cite{GLR-book, GLR-book2}, and studied the (classical) sign characteristic of its linearizations in the vector space $\mathbb{DL}(P)$ and in the family of 
block-symmetric generalized Fiedler pencils with repetition  (introduced in \cite{BDFM}),  as well as of the
linearizations $\mathcal{T}_P(\lambda)$ and $\mathcal{R}_P(\lambda)$, to be introduced in Section \ref{GFP} (see also the Introduction). Besides these papers and \cite{ammari}, where some partial results on the sign characteristic of  linearizations in $\mathbb{DL}(P)$  are given, the authors do not know of any other papers in the literature studying the sign characteristic of specific linearizations of matrix polynomials.   (We observe that a  generalization of the concept of sign 
characteristic to general matrix polynomials was published recently \cite{SC-new} but we are not aware of any results regarding
 the sign characteristic of Hermitian linearizations, using this more general definition). Since a scaling of $P(\lambda)$ may be used
 to improve conditioning or backward errors of  its linearizations, we note in the following theorem that the procedure of scaling a Hermitian 
 $P(\lambda)$ with nonsingular leading coefficient  preserves the sign characteristic of $P(\lambda)$, when the scaling parameters $\beta$ and $\gamma$ are positive real numbers.  Based on this result, the (classical) sign characteristic of the 
 original matrix polynomial can be obtained from the sign characteristic of a linearization $L(\mu)$ of the scaled matrix polynomial
 $\tilde{P}(\mu)$, as long as we know how the sign characteristic of $\tilde{P}(\mu)$ and $L(\mu)$ are related. The proof of the next theorem appears in Appendix A.

\begin{theorem}\label{SC-preserv}
Let $P(\lambda)$ be a Hermitian matrix polynomial of degree $k$ as in (\ref{pol}), with
$A_{k}$ nonsingular, and let $\widetilde{P}(\mu):=\beta P(\gamma\mu),$ where
$\beta$ and $\gamma$ are positive real numbers. Then, $P(\lambda)$ and
$\widetilde{P}(\mu)$ have the same  (classical) sign characteristic.
\end{theorem}


\section{Some definitions and auxiliary  results}\label{sec-aux}

We  introduce some  concepts and technical results that will be used in the proofs of our main results. 

If  $a$ and $b$ are two positive integers such that $a\leq b$, we define
$$a:b := a, a+1, \ldots, b.$$
The following result is an immediate consequence of the Cauchy-Schwarz inequality when the standard inner product is considered in $\mathbb{C}^n$.
\begin{lemma}\label{tech1}
Let $m$ be a positive integer and let $a$ be a positive real number. Then,
\begin{align*}
\left(\sum_{j=0}^{m}a^j\right)^2 \leq (m+1) \sum_{j=0}^{m} a^{2j}.
\end{align*}
\end{lemma}


The following property is well known (see Lemma 3.5 in \cite{backward} for  a proof of the second inequality). 

\begin{proposition}\label{prop}
For any complex $l\times m$ block-matrix $B=(B_{ij})$ we have 
\begin{equation}\label{bound-norm}
\max_{i,j} \|B_{ij}\|_2 \leq \|B\|_2 \leq \sqrt{lm}\;  \max_{i,j} \|B_{ij}\|_2.
\end{equation}
\end{proposition}

Given a matrix polynomial $P(\lambda)$ of degree  $k$ as in (\ref{pol}), the \emph{$i$th Horner shift of $P(\lambda)$},  $i=0:k$, is given by
\begin{equation}\label{pol-Pi}
P_{i}(\lambda):=\lambda^{i}A_{k}+\lambda^{i-1}A_{k-1}+\cdots+\lambda
A_{k-i+1}+A_{k-i}.
\end{equation}
Notice that
$P_{0}(\lambda)=A_{k}$,  $P_{k}(\lambda)=P(\lambda)$, and 
\begin{equation}\label{horner}
P_{i+1}(\lambda)-A_{k-i-1}=\lambda P_{i}(\lambda), \quad i=0:k-1.
\end{equation}
 When convenient, we  write $P_i$ to denote $P_i(\lambda)$. 
  We also denote 
  \begin{equation}\label{pol-Pii}
P^i(\lambda):=\lambda^i A_i + \cdots + \lambda A_1 + A_0, \quad i=0:k.
\end{equation}
Notice that $P^0(\lambda)=A_0$ and $P^k(\lambda)=P(\lambda)$.


\begin{lemma}\label{PrPr}
Let $P(\lambda)$ be a regular matrix polynomial of degree $k$ as in (\ref{pol}).  Let $P_i(\lambda)$ and $P^i(\lambda)$, $i=0:k$, be the matrix polynomials defined in (\ref{pol-Pi}) and (\ref{pol-Pii}). Let $\delta$ be a nonzero, finite  eigenvalue of $P(\lambda)$, and let $x$  and $y$  be, respectively,  a right  and left eigenvector of $P(\lambda)$ associated with $\delta$. Then,
$$P_i(\delta) x = -\delta^{i-k} P^{k-i-1}(\delta) x, \ \quad  \textrm{and} \quad  y^* P_i(\delta)= -\delta^{i-k} y^* P^{k-i-1}(\delta), \quad i=0:k-1.$$
\end{lemma}

\begin{proof}
Note that, for $i=0:k-1$, we have $P(\delta)=\delta^{k-i} P_i (\delta)+ P^{k-i-1}(\delta)$. Thus, the result follows taking into account that $\delta$ is nonzero,  
$P(\delta)x=0$, and $y^* P(\delta)=0$, since  $x$ and $y$ are, respectively,   a right  and a left eigenvector of $P(\lambda)$ associated with $\delta$.
\end{proof}

The next lemma can be easily verified.

\begin{lemma}\label{Pjibound}
Let $P(\lambda)$ be a  matrix polynomial of degree $k$ as in (\ref{pol}), let $\delta \in \mathbb{C}$, and let $P_i(\lambda)$ and $P^i(\lambda)$, $i=0:k$, be the matrix polynomials defined in (\ref{pol-Pi}) and (\ref{pol-Pii}). Then, 
 $$\|P_i(\delta)\|_2, \|P^i(\delta)\|_2 \leq \max_{j=0:k}\{\|A_j\|_2 \} \sum_{j=0}^i |\delta|^j, \quad i=0:k.$$
\end{lemma}

We close this section with two combinatorial lemmas that will be used later in the proofs of our main results.

\begin{lemma}\label{Deltax}
Let $P(\lambda)$ be an  $n\times n$  matrix polynomial  of odd degree $k$ as in (\ref{pol}), let $P_i(\lambda)$, $i=0:k$, be the Horner shifts defined in (\ref{pol-Pi}),   and let 
\begin{equation}\label{Delta}
\Delta^{\mathcal{B}}(\lambda): = [ \lambda^{\frac{k-1}{2}} I_n, \lambda^{\frac{k-1}{2}} P_1, \lambda^{\frac{k-3}{2}} I_n, \lambda^{\frac{k-3}{2}} P_3, \ldots, \lambda I_n, \lambda P_{k-2}, I_n],
\end{equation}
where $\Delta^{\mathcal{B}}(\lambda)$ denotes the block-transpose of $\Delta(\lambda)$ when viewed as a $k\times 1$ block-matrix whose blocks are $n\times n$.  Let $\delta \in \mathbb{C}$. Then, 
\begin{equation}\label{boundDelta}
\|\Delta(\delta)\|_2\leq \sqrt{d_1(\delta)}\;  \max_{i=0:k} \{1, \|A_i\|_2\},
\end{equation}
where 
\begin{equation}\label{Delta1}
d_1(\delta)=\sum_{r=0}^{\frac{k-1}{2}} |\delta|^{2r} + \sum_{r=1}^{\frac{k-1}{2}} \left( (k-2r+1) \sum_{s=r}^{k-r} |\delta|^{2s} \right).
\end{equation}

\end{lemma}

\begin{proof}
Let $\delta \in \mathbb{C}$ and  $0\neq x \in \mathbb{C}^n$. 
Taking into account the definition of $\Delta(\lambda)$, we get
\begin{align}
\|\Delta(\delta) x\|_2^2& =\sum_{r=0}^{\frac{k-1}{2}} |\delta|^{2r} \|x\|_2^2 + \sum_{r=1}^{\frac{k-1}{2}} |\delta|^{2r} \|P_{k-2r}(\delta)x\|_2^2 \label{first-line}\\
& \leq \left[ \sum_{r=0}^{\frac{k-1}{2}} |\delta|^{2r}  + \sum_{r=1}^{\frac{k-1}{2}} |\delta|^{2r} \|P_{k-2r}(\delta)\|_2^2\right] \|x\|_2^2.
\end{align}

By Lemmas \ref{Pjibound} and \ref{tech1}, we obtain
\begin{align*}
\left(\frac{\|\Delta(\delta) x\|_2}{\|x\|_2}\right)^2 &\leq 
\sum_{r=0}^{\frac{k-1}{2}} |\delta|^{2r}  + \sum_{r=1}^{\frac{k-1}{2}} |\delta|^{2r} \left( \max_{i=0:k}\{\|A_i\|_2\} \sum_{s=0}^{k-2r} |\delta|^s\right)^2 \\
&\leq \left[ \sum_{r=0}^{\frac{k-1}{2}} |\delta|^{2r}  + \sum_{r=1}^{\frac{k-1}{2}} \left( (k-2r+1) \sum_{s=r}^{k-r} |\delta|^{2s}\right) \right]\max_{i=0:k} \{1, \|A_i\|_2\}^2\\
&= d_1(\delta) \max_{i=0:k} \{\|A_i\|_2,1\}^2.
\end{align*}
Since $\|\Delta(\delta)\|_2 = sup_{x\neq 0} \frac{\|\Delta(\delta)x\|_2}{\|x\|_2},$  (\ref{boundDelta}) follows. \end{proof}

\begin{lemma}\label{d1bound}
Let  $\delta \in \mathbb{C}$ be nonzero and let $k\geq 3$ be a positive odd  integer. Let $d_1(\lambda)$  be as in (\ref{Delta1}). 
If $|\delta| \leq 1$, then 
\begin{equation}\label{d1boundeq}
 d_1(\delta) \leq \frac{k+1}{2}+\frac{(k-1)^3}{2} |\delta|^{2}.
\end{equation}
\end{lemma}

\begin{proof}
Assume that $|\delta| <1$. Then, 
\begin{align*}
\sum_{r=1}^{\frac{k-1}{2}} \left( (k-2r+1) \sum_{s=r}^{k-r} |\delta|^{2s} \right) &\leq    \frac{(k-1)^2}{2}  \sum_{s=1}^{k-1} |\delta|^{2s} \\
&= \frac{(k-1)^2}{2}|\delta|^2 \sum_{s=1}^{k-1} |\delta|^{2(s-1)} \leq \frac{(k-1)^3}{2} |\delta|^{2}.
 \end{align*}
 Thus, from (\ref{Delta1}), (\ref{d1boundeq}) follows.
 The case $|\delta|=1$ follows by using a continuity argument. 
\end{proof}



\section{Block-symmetric block-tridiagonal generalized Fiedler pencils} \label{GFP}

In this section we present the block-symmetric  linearizations of a matrix polynomial $P(\lambda)$ that are the focus of  this paper.  Here and in  the next section  we  only consider matrix polynomials $P(\lambda)$ of odd degree $k$ and we assume that $k >2$.  

The family of generalized Fiedler pencils (GFP) associated with a matrix polynomial $P(\lambda)$ was introduced in \cite{ant-vol04, bddsiam}.
The following block-symmetric block-tridiagonal GFP associated with an $n\times n$ $P(\lambda)$ of odd degree $k$  was presented in \cite[Theorem 3.1]{ant-vol04}:
\begin{equation}
\mathcal{T}_P(\lambda)=%
\begin{bmatrix}
\lambda A_{k}+A_{k-1} & -I_n &  &  &  &  &  & 0\\
-I_n & 0 & \lambda I_n &  &  &  &  & \\
& \lambda I_n & \lambda A_{k-2}+A_{k-3} & -I_n &  &  &  & \\
&  & -I_n & 0 &  &  &  & \\
&  &  &  & \ddots &  &  & \\
&  &  &  &  & \lambda A_{3}+A_{2} & -I_n & \\
&  &  &  &  & -I_n & 0 & \lambda I_n\\
0  &  &  &  &  &  & \lambda I_n & \lambda A_{1}+A_{0}%
\end{bmatrix}.
\label{GFPT}%
\end{equation}

Also, the following block-symmetric block-tridiagonal pencil associated with an $n\times n$ $P(\lambda)$ of odd degree $k$  was introduced in \cite[formula (5.1)]{4m-alt} :
\begin{equation}\label{GFP-R}
\mathcal{R}_P(\lambda ):= S R\mathcal{T}_P (\lambda)RS\end{equation}
where  
\begin{equation}\label{flipR}
R:=\left[ \begin{array}{ccc} & & I_n \\ & \iddots & \\ I_n & & \end{array}\right] \in \mathbb{C}^{nk \times nk}
\end{equation}
and $S$ is the $k\times k$  block-diagonal matrix whose $(i,i)$th block-entry is given by 
\begin{equation}\label{def-S}
S(i, i)=\left\{\begin{array}{rl} -I_n, & \textrm{if $i \equiv 0, 1$ mod 4,}\\ I_n, & \textrm{otherwise.} \end{array}\right.
\end{equation}


It is well-known \cite{ant-vol04, 4m-alt} that, for any matrix polynomial $P(\lambda)$ of odd degree $k$,  the pencils $\mathcal{T}_P(\lambda)$  and $\mathcal{R}_P(\lambda)$ are strong linearizations of  $P(\lambda)$.  
As mentioned in the introduction, these linearizations have  several attractive properties. In particular, 
 it is easy to recover an eigenvector of $P(\lambda)$ associated with an eigenvalue $\delta$ from an eigenvector of the linearizations associated with the same eigenvalue, as we show next. We start with a  technical  lemma that will be useful for this purpose. 

Here and in the next sections, we denote by $e_i$ the $i$th column of the identity matrix of appropriate size for the context. 


\begin{lemma}\label{GFP-ans}
Let $P(\lambda)$ be a matrix polynomial of odd degree $k$ as in (\ref{pol}) and let $\mathcal{T}_P(\lambda)$  be as in (\ref{GFPT}). Then,
\begin{equation}\label{T-ans}
\mathcal{T}_P(\lambda) \Delta(\lambda) = e_k \otimes P(\lambda), \quad \textrm{and} \quad 
\Delta^{\mathcal{B}}(\lambda) \mathcal{T}_P(\lambda) = e_k^T \otimes P(\lambda),
\end{equation}
where $\Delta^{\mathcal{B}}(\lambda)$ is   as in (\ref{Delta}). 

\end{lemma}

\begin{proof}
Let $\mathcal{T}_P(\lambda):=\lambda \mathcal{T}_1 - \mathcal{T}_0$.  Using  (\ref{horner}), a direct computation shows that 
$$\mathcal{T}_1 \Delta(\lambda)= \left[ \begin{array}{c} \lambda^{\frac{k-1}{2}} P_0 \\ \lambda^{\frac{k-1}{2}-1}  I_n \\ \hline  \lambda^{\frac{k-1}{2}-1} P_2 \\ \lambda^{\frac{k-1}{2}-2} I_n \\ \hline \vdots \\\hline  \lambda P_{k-3} \\   I_n \\ \hline P_{k-1} \end{array} \right]\quad \textrm{and} \quad \mathcal{T}_0 \Delta(\lambda)= \left[ \begin{array}{c} \lambda^{\frac{k+1}{2}} P_0 \\ \lambda^{\frac{k-1}{2}}  I_n \\ \hline  \lambda^{\frac{k-1}{2}} P_2 \\ \lambda^{\frac{k-1}{2}-1} I_n \\ \hline \vdots \\\hline  \lambda^2 P_{k-3} \\  \lambda I_n \\ \hline -A_0 \end{array} \right],$$ 
where $P_i$, $i=0:k$, are the Horner shifts defined in (\ref{pol-Pi}). 

Taking into account (\ref{horner}) again and the fact that $P_k(\lambda)=P(\lambda)$,  the first claim in (\ref{T-ans}) follows. 
The second claim in  (\ref{T-ans})  follows easily from the first claim  by noting that 
$(\mathcal{T}_P(\lambda) \Delta(\lambda))^{\mathcal{B}} = \Delta^{\mathcal{B}}(\lambda) \mathcal{T}_P(\lambda)$, as the $i$th block-row of $\mathcal{T}_P(\lambda)$, with $i$ even, just 
contains blocks of the form $0,$ $I_n$, and $\lambda I_n$, and this type of blocks  commute with $P_j$. 
\end{proof}

 Note that  the equations  in  (\ref{T-ans}) are a  particular case of equation (2.11) in \cite{backward} (where the homogeneous approach is considered).

The following theorem follows from Lemma \ref{GFP-ans} using arguments similar to those in the proof of Theorem 3.8 in \cite{4m-vspace}. 

\begin{theorem}\label{eig-T}
Let $P(\lambda)$ be a  regular matrix polynomial of odd degree $k$ as in (\ref{pol}). Assume that  $\delta$ is a finite eigenvalue of $P(\lambda)$. Let $\Delta(\lambda)$ be as in (\ref{Delta}).
A vector $z$ is a right (resp. left) eigenvector of $\mathcal{T}_P(\lambda)$ associated with $\delta$  if and only if $z=\Delta(\delta) x$ (resp. $z=\overline{\Delta(\delta)}  y$),  for some right (resp. left) eigenvector $x$ (resp. $y$) of $P(\lambda)$ associated with $\delta$.  
\end{theorem}

\begin{proof}
Taking into account Lemma \ref{GFP-ans}, we have
\[
\mathcal{T}_{P}(\lambda)\Delta(\lambda)x=(e_{k}\otimes P(\lambda
))x=e_{k}\otimes P(\lambda)x.
\]

Clearly, if $\{x_{1},\ldots,x_{m}\}$ is a basis of the eigenspace associated with
the eigenvalue $\delta$ of $P(\lambda)$, then $\{\Delta(\delta)x_{1}
,\ldots,\Delta(\delta)x_{k}\}$ is linearly independent and, since $\delta$ has
the same geometric multiplicity as an eigenvalue of $P(\lambda)$ and as eigenvalue of
$\mathcal{T}_{P}(\lambda),$ because $\mathcal{T}_{P}(\lambda)$ is a linearization
of $P(\lambda)$, then $\{\Delta(\delta)x_{1},\ldots,\Delta(\delta)x_{k}\}$ is 
a  basis of the eigenspace associated with the eigenvalue $\delta$ of
$\mathcal{T}_{P}(\lambda).$ Thus, a vector $z$ is a right eigenvector of $\mathcal{T}_{P}
(\lambda)$ associated with $\delta$ if and only if it is a linear combination of $\Delta
(\delta)x_{1},\ldots,\Delta(\delta)x_{k}$, that is, it is of the form
$\Delta(\delta)x$ for some eigenvector $x$ of $P(\lambda)$ associated with
$\delta.$ 

The proof for left eigenvectors is similar.\end{proof}


Let $R$ be as in (\ref{flipR}), $S$ be as in (\ref{def-S}), and $\Delta(\lambda)$ be as  in (\ref{Delta}). From (\ref{GFP-R}) and Lemma \ref{GFP-ans}, we obtain
\begin{equation}\label{R-ans}
\mathcal{R}_P(\lambda)(SR \Delta(\lambda)) = (SR e_k) \otimes P(\lambda).
\end{equation}
Note that  $R^2=I$ and  $S^2=I$.
Based on this fact, we give next an expression for the right and left eigenvectors of $\mathcal{R}_P(\lambda)$.

\begin{theorem}\label{Reig}
Let $P(\lambda)$ be a regular matrix polynomial of odd degree $k$ as in (\ref{pol}). Assume that $\delta$ is a finite eigenvalue of $P(\lambda)$. Let $\Delta(\lambda)$ be as in (\ref{Delta}).
A vector $z$ is a right (resp. left) eigenvector of $\mathcal{R}_P(\lambda)$ associated with $\delta$ if and only if $z=SR \Delta(\delta) x$ (resp.
$z=SR \overline{\Delta(\delta)} y$), for some right  (resp. left) eigenvector $x$ (resp. $y$) of $P(\lambda)$ associated with $\delta$. 
\end{theorem}



We close  this section with two results that  will help us studying the  numerical performance of the linearizations $\mathcal{T}_P(\lambda)$ and $\mathcal{R}_P(\lambda)$ in the next section. 
The next lemma, which can be easily verified, allows us to   focus on eigenvalues of modulus not greater than $1.$

\begin{lemma}
\label{LTrevP}Let $P(\lambda)$ be a matrix polynomial of odd degree $k$ as in
(\ref{pol}) with $A_0 \neq 0$. Then,
\[
\mathcal{T}_{P}(\lambda)=\lambda DR\mathcal{T}_{revP}\left(  \frac{1}{\lambda
}\right)  RD,
\]
where $R$ is as in (\ref{flipR})\ and
\begin{equation}
D=diag(I_{n},-I_{n},I_{n},-I_{n},\ldots,I_{n}).\label{DD}%
\end{equation}
Moreover, $z\in\mathbb{C}^{nk}$ is a right (resp. left) eigenvector of
$\mathcal{T}_{P}(\lambda)$ associated with a nonzero eigenvalue $\delta$ if
and only if $RDz$ is a right (resp. left) eigenvector of $\mathcal{T}_{revP}%
(\lambda)$ associated with the eigenvalue $\frac{1}{\delta}.$
\end{lemma}

The following result gives an upper bound on the spectral norm of the matrix coefficients of $\mathcal{T}_P(\lambda)$, improving the one  obtained  by using  Proposition \ref{prop}.

\begin{proposition}\label{boundTp}
Let  $P(\lambda)=\sum_{i=0}^k A_i \lambda^i $ be an $n\times n$ matrix polynomial of odd degree $k$ and $\mathcal{T}_P(\lambda) := \lambda \mathcal{T}_1 - \mathcal{T}_0$ be as in (\ref{GFPT}).  Then,
$$\|\mathcal{T}_1\|_2 , \|\mathcal{T}_0\|_2 \leq 2\; \max_{i=0:k} \{1, \|A_i \|_2\}.$$

\end{proposition}

\begin{proof}
Let $z=[z_1, \ldots, z_k]^{\mathcal{B}}$ be a nonzero $k \times 1$ block-vector partitioned into $n\times 1$ blocks. Then, defining $z_0:=0$, we have
\begin{align*}
\|\mathcal{T}_1 z \|_2^2 & = \sum_{i=1}^{\frac{k-1}{2}}  \|z_{2i+1}\|_2^2+ \sum_{i=0}^{\frac{k-1}{2}} \| z_{2i} +A_{k-2i} z_{2i+1} \|_2^2\\
& \leq \sum_{i=1}^{\frac{k-1}{2}}  \|z_{2i+1}\|_2^2+ \sum_{i=0}^{\frac{k-1}{2}}\left ( \|z_{2i}\|_2^2 + \|A_{k-2i} z_{2i+1}\|_2^2 + 2 \|z_{2i}\|_2\|A_{k-2i}z_{2i+1}\|_2 \right )\\
& \leq \max_{i=0:k}\{1, \|A_i\|_2^2\} \left( \sum_{i=1}^{k}  \|z_i\|_2^2+ \sum_{i=0}^{\frac{k-1}{2}} \left( \|z_{2i+1}\|_2^2 + 2 \max\{ \|z_{2i}\|_2^2, \|z_{2i+1}\|_2^2\}\right) \right)\\
& \leq \max_{i=0:k}\{1, \|A_i\|_2^2\} (2\|z\|_2^2 +  2 \sum_{i=0}^{\frac{k-1}{2}} \max\{\|z_{2i}\|_2^2, \|z_{2i+1}\|_2^2\}) \leq 4 \max_{i=0:k}\{1, \|A_i\|_2^2\} \|z\|_2^2.
\end{align*}
The proof for $\mathcal{T}_0$ is analogous.
\end{proof}

\section{Conditioning and backward error of $\mathcal{T}_P(\lambda)$ and $\mathcal{R}_P(\lambda)$}\label{main}

Let $P(\lambda)$ be a matrix polynomial of odd degree $k$ with $A_0\neq 0$ and let $\mathcal{T}_P(\lambda)$ and $\mathcal{R}_P(\lambda)$ be the linearizations of $P(\lambda)$ defined in (\ref{GFPT}) and (\ref{GFP-R}), respectively. 
In this section we present the main results in this paper, Theorems \ref{conditioningR} and \ref{backwardR}, concerned with the comparison of the conditioning of eigenvalues and the backward error of approximate eigenpairs of  $\mathcal{T}_P(\lambda)$  with, respectively, the conditioning of eigenvalues and the backward error of approximate eigenpairs of $P(\lambda)$.  As will follow from Remarks \ref{cond-Rp} and \ref{back-Rp},   Theorems \ref{conditioningR} and \ref{backwardR}  apply to $\mathcal{R}_P(\lambda)$ as well.


We next state Theorems \ref{conditioningR} and \ref{backwardR}, which  will be proven in Sections \ref{conditioning;sec} and \ref{sec;backward}, respectively. We use the notation for $P(\lambda)$ in (\ref{pol}) and we denote

\begin{equation}\label{def-rho}
\rho_{1}:=\frac{\max_{i=0:k}\{1,\Vert A_{i}\Vert_{2}^3\}}%
{\min\{\Vert A_{k}\Vert_{2},\Vert A_{0}\Vert_{2}\}}, \quad \rho_{2}:=\frac
{\min\{\max\{1,\Vert A_{k}\Vert_{2}\},\max\{1,\Vert A_{0}\Vert_{2}%
\}\}}{\max_{i=0:k}\{\Vert A_{i}\Vert_{2}\}}
\end{equation}
and 
\begin{equation}
\rho^{\prime}:=\frac{\max_{i=0:k}\{1,\Vert A_{i}\Vert_{2}^2\}
}{\min\{\Vert A_{k}\Vert_{2},\Vert A_{0}\Vert_{2}\}}.\label{def-rho2}%
\end{equation}



\begin{theorem}
\label{conditioningR}Let $P(\lambda)$ be a  regular matrix polynomial of odd degree $k$
as in (\ref{pol}) with $A_0\neq 0$. Assume that  $\delta$ is a simple, finite, nonzero eigenvalue of
$P(\lambda)$.   Let $\rho_1$ and $\rho_2$ be as in (\ref{def-rho}). Then, 

\begin{equation}\label{Tcondup}
 \rho_2 \leq \frac{\kappa
_{\mathcal{T}_P}(\delta)}{\kappa_{P}(\delta)}  \leq  2k^3 \rho_1.
\end{equation}
Moreover, if $|\delta| \neq 1$ and  $\min\{ |\delta|, \frac{1}{|\delta|} \}\leq  \frac{1}{k-1}$, then 
$$\frac{\kappa
_{\mathcal{T}_P}(\delta)}{\kappa_{P}(\delta)} \leq 4(k+1) \rho_1.$$

\end{theorem}

\begin{remark}\label{cond-Rp}
Taking into account the definition of $\mathcal{R}_{P}(\lambda)$ and Theorem
\ref{Reig}, it follows easily from Theorem \ref{condform} applied to
$\mathcal{T}_{P}(\lambda)$ and $\mathcal{R}_{P}(\lambda)$ that $\kappa_{\mathcal{T}_{P}}%
(\delta)=\kappa_{\mathcal{R}_{P}}(\delta).$ Recall that  
the spectral norm is unitarily invariant. Thus, Theorem \ref{Tcondup} holds for $\mathcal{R}_P(\lambda)$ as well. 
\end{remark}

From Theorem \ref{conditioningR} we can conclude that, if  the norms of the matrix coefficients of $P(\lambda)$ have similar magnitudes and  $\max_{i=0:k}  \{\|A_i\|_2\} =1$ (which can be obtained by scaling $P(\lambda)$ as explained in Section \ref{cone-sec}), the condition number of any simple, nonzero, finite eigenvalue of $P(\lambda)$ is close to the condition number of the same eigenvalue of $\mathcal{T}_P(\lambda)$ { \it with no restriction on the modulus of $\delta$}, in contrast with what happens to $D_1(\lambda, P)$ and $D_k(\lambda, P)$. 

Since, for not extremely large values of $nk$,  the algorithm QZ, combined with adequate methods for computing eigenvectors \cite{moler},  is used to compute the eigenvalues and eigenvectors of a linearization of $P(\lambda)$ and this algorithm produces small backward errors of order unit-roundoff, if we prove that  $\eta_P(x, \delta)$ is not much larger than $\eta_{\mathcal{T}_P}(z, \delta)$ (where $(z, \delta)$ denotes an approximate eigenpair of $\mathcal{T}_P$ and $(x, \delta)$ denotes an approximate eigenpair of $P$ obtained from $(z, \delta)$ in an appropriate way),  we ensure small backward errors for the approximate eigenpairs of $P(\lambda)$ as well. This motivates the next theorem, in which an upper bound for the ratio $ \frac{\eta_P(x, \delta)}{\eta_{\mathcal{T}_P}(z, \delta)} $  is established.

\begin{theorem}\label{backwardR}
Let $P(\lambda)$ be a  matrix polynomial  of odd degree $k$ as in (\ref{pol}) with $A_0 \neq 0$.  Let  $(z, \delta)$ be an approximate right eigenpair of $\mathcal{T}_P(\lambda)$,
 $ x=(e_1^T \otimes I_n)z$ if $|\delta| > 1$ and  $x=(e_k^T \otimes I_n)z$ if $|\delta| \leq 1$.   Then, $(x, \delta)$ is an approximate right eigenpair of $P(\lambda)$  and 
\begin{equation}\label{back-bound}
 \frac{\eta_P(x, \delta)}{\eta_{\mathcal{T}_P}(z, \delta)} \leq 4 k^{\frac{3}{2}}\frac{\|z\|_2}{\|x\|_2} \rho'.
\end{equation}
Moreover, if $|\delta| \neq 1$ and $\min \{ |\delta|, \frac{1}{|\delta|} \}\leq \frac{1}{k-1}$, then 
\begin{equation}\label{second-back-T}
 \frac{\eta_P(x, \delta)}{\eta_{T_P}(z, \delta)} \leq  4 \sqrt{k+1} \frac{\|z\|_2}{\|x\|_2}\rho'. 
\end{equation}
\end{theorem}

\begin{remark}\label{back-Rp} Theorem \ref{backwardR} also holds for $\mathcal{R}_P(\lambda)$ taking into consideration the following observations. 
Let $(z,\delta)$ be an approximate
eigenpair of $\mathcal{R}_{P}(\lambda)$. It is easy to see that
\[
\eta_{\mathcal{R}_{P}}(z,\delta)=\eta_{\mathcal{T}_{P}}(RSz,\delta).
\]
Since, for $t\in\{1,k\}$, $(e_{t}^{T}\otimes I_{n})z=\pm(e_{k-t+1}^{T}\otimes I_{n})RSz:=\pm x_t$,  we have
\[
\frac{\eta_{P}(( e_t^T \otimes I_n)z,\delta)}{\eta_{\mathcal{R}_{P}}(z,\delta)}=\frac
{\eta_{\mathcal{P}}(x_t,\delta)}{\eta_{\mathcal{T}_{P}}(RSz,\delta)}.
\]
Thus, it is clear that, if proven  for
$\mathcal{T}_{P}(\lambda)$, Theorem \ref{backwardR} also holds for $\mathcal{R}_P(\lambda)$. 
\end{remark}

In Theorem \ref{backwardR} we considered right eigenpairs of $P(\lambda)$ and  $\mathcal{T}_P(\lambda)$. The case of left eigenpairs can be easily reduced to the previous case. Indeed, $(y^*, \delta)$ is a left eigenpair of $P(\lambda)$ if and only if $(y, \overline{\delta})$ is a right eigenpair of $P^*(\lambda)$. Similarly, $(z^*, \delta)$ is a left eigenpair of $\mathcal{T}_P(\lambda)$ if and only if $(z, \overline{\delta})$ is a right eigenpair of $\mathcal{T}^*_P(\lambda)= \mathcal{T}_{P^*}(\lambda).$ Thus, it can be easily seen that 
$$ \eta_P(y^*, \delta)= \eta_{P^*}(y, \overline{\delta}) \quad \textrm{and} \quad \eta_{\mathcal{T}_P}(z^*, \delta)= \eta_{\mathcal{T}_{P^*}}(z, \overline{\delta}),$$
and the result for approximate left eigenpairs follows from Theorem \ref{backwardR} applied to $P^*(\lambda)$ and $\mathcal{T}_{P^*}(\lambda)$. 

Note that Theorems \ref{conditioningR} and \ref{backwardR} hold with no nonsingularity restrictions on the matrix coefficients  of $P(\lambda)$, in contrast with the analogous results  for the block-symmetric linearizations $D_1(\lambda, P)$ and $D_k(\lambda, P)$ (see \cite{backward, tisseur} and Section \ref{cond-D1Dk}). 

\begin{remark}\label{quotientzx}
From Theorem \ref{eig-T}, $x^{\prime}$ is an eigenvector of $P(\lambda)$
associated with the eigenvalue $\delta\ $if and only if $z=\Delta
(\delta)x^{\prime}$ is an eigenvector of $\mathcal{T}_{P}(\lambda)$
associated with $\delta.$ As in the proof of Lemma \ref{Deltax}, for
$z=\Delta(\delta)x^{\prime}$ (which implies $x^{\prime}=(e_{k}^{T}\otimes
I_{n})z$) , we obtain%
\[
\frac{\Vert z\Vert_{2}^{2}}{\Vert x^{\prime}\Vert_{2}^{2}}\leq d_{1}%
(\delta)\max_{i=0:k}\{\Vert A_{i}\Vert_{2},1\}^{2},
\]
where $d_{1}(\delta)$ is as in (\ref{Delta1}). If $|\delta|\leq1,$ from Lemma
\ref{d1bound},  we have
\[
\frac{\Vert z\Vert_{2}^{2}}{\Vert x^{\prime}\Vert_{2}^{2}}\leq\left(
\frac{k+1}{2}+\frac{(k-1)^{3}}{2}\right)  \max_{i=0:k}\{1,\Vert A_{i}\Vert
_{2}\}^2,
\]
implying%
\[
\frac{\Vert z\Vert_{2}}{\Vert x^{\prime}\Vert_{2}}\leq\frac{1}{2}k^{\frac
{3}{2}}\max_{i=0:k}\{1,\Vert A_{i}\Vert_{2}\}.
\]
If $|\delta|>1,$ taking into account Lemma \ref{LTrevP}, the eigenvectors of
$\mathcal{T}_{revP}(\lambda)$ associated with the eigenvalue $\frac{1}{\delta
}$ are of the form $RDz.$ From Theorem \ref{eig-T}, $(e_{1}^{T}\otimes
I_{n})RDz$ is an eigenvector of $revP(\lambda)$ associated with the eigenvalue
$\frac{1}{\delta}\ $and, from the calculations above, we also have
\[
\frac{\Vert z\Vert_{2}}{\Vert(e_{1}^{T}\otimes I_{n})z\Vert_{2}}=\frac{\Vert
RDz\Vert_{2}}{\Vert(e_{k}^{T}\otimes I_{n})RDz\Vert_{2}}\leq\frac{1}%
{2}k^{\frac{3}{2}}\max_{i=0:k}\{1,\Vert A_{i}\Vert_{2}\}.
\]
Thus, if 
$P(\lambda)$ is scaled so that $\max_{i=0:k}\{\Vert A_{i}\Vert_{2}\}=1,$ we
obtain
\[
\frac{\Vert z\Vert_{2}}{\Vert x\Vert_{2}}\lessapprox\frac{1}{2}k^{\frac{3}{2}%
},
\]
where $x=(e_{k}^{T}\otimes I_{n})z$ if $|\delta|\leq1$ and $x=(e_{1}%
^{T}\otimes I_{n})z$ if $|\delta|>1.$ If we assume that an approximate  eigenvector $z$ has a block structure similar to that of an eigenvector of $\mathcal{T}_p(\lambda)$, then we can expect  that the bound above for the quotient $\frac{\|z\|_2}{\|x\|_2}$ appearing in Theorem \ref{backwardR}  still holds and it is close to 1 for moderate $k$.  

Note that, when the matrix coefficients of $P(\lambda)$ have similar norms and
$P(\lambda)$ is scaled so that $\max_{i=0:k}\{\Vert A_{i}\Vert_{2}\}=1$, then
$\rho^{\prime}\approx1$ and the bound of the quotient of backward errors in
Theorem \ref{backwardR} is expected to depend only on $k$. 
\end{remark}


In the rest of this section we use the notation introduced in Section \ref{GFP}. In particular, $\mathcal{T}_P(\lambda):=\lambda \mathcal{T}_1 -\mathcal{T}_0$ denotes the block symmetric pencil defined in (\ref{GFPT}).

\subsection{Proof of Theorem \ref{conditioningR}}\label{conditioning;sec}
 We start with a lemma in which we give an explicit expression for the condition number of a simple, finite, nonzero eigenvalue $\delta$ of the  linearization $\mathcal{T}_P(\lambda)$  of a matrix polynomial $P(\lambda)$ of odd degree.

\begin{lemma}\label{kappatp}
Let $P(\lambda)$ be a regular matrix polynomial of odd degree $k$ as in (\ref{pol}). Assume that $\delta$ is a simple, finite, nonzero
eigenvalue of $P(\lambda)$ with left and right eigenvectors $x_1$ and $x_2$, respectively.  Let 
$\mathcal{T}_P(\lambda):=\lambda \mathcal{T}_1 - \mathcal{T}_0$ 
Then, 
\begin{equation}\label{kT}
\kappa_{\mathcal{T}_P} (\delta)= \frac{(|\delta| \|\mathcal{T}_1 \|_2+ \|\mathcal{T}_0\|_2) \|\Delta(\delta) x_1\|_2 \|\Delta(\delta) x_2\|_2}{|\delta| |x_1^* P'(\lambda) x_2|},
\end{equation}
where $\Delta(\lambda)$ is defined as in (\ref{Delta}).


\end{lemma}

\begin{proof}
 Differentiating  the first equality in (\ref{T-ans}), we get
\begin{equation}\label{GFP-eq1}
\mathcal{T}'(\lambda)\Delta(\lambda) + \mathcal{T}(\lambda) \Delta'(\lambda) = e_k \otimes P'(\lambda).
\end{equation}
By Theorem \ref{eig-T},  the vector   $z_1= \overline{\Delta}(\delta)  x_1$ is a left eigenvector of $\mathcal{T}_P(\lambda)$ associated with $\delta$ and $z_2=\Delta(\delta)  x_2$  is a right eigenvector of $\mathcal{T}_P(\lambda)$ associated with $\delta$.  
Evaluating the  expression (\ref{GFP-eq1}) at $\delta$, premultiplying  by  $z_1^*$, and postmultiplying  by $ x_2$, we get
$$z_1^* \mathcal{T}'(\delta)(\Delta(\delta)  x_2) = z_1^* (e_{k} \otimes P'(\delta)x_2),$$
or, equivalently, 
$$z_1^* \mathcal{T}'(\delta) z_2 = x_1^* \Delta^T(\delta)(e_k \otimes P'(\delta)x_2)=x_1^* P'(\delta) x_2.$$
Now (\ref{kT}) follows from Theorem \ref{condform}.
\end{proof}

The following lemma will allow  us to only consider eigenvalues $\delta$ such that $|\delta| \leq 1$ when proving Theorem \ref{conditioningR}. 

\begin{lemma}
\label{condPrevP}Let $P(\lambda)$ be a regular matrix polynomial of odd degree
$k$ as in (\ref{pol}) with $A_0 \neq 0$, and let $\delta$ be  a simple, finite, nonzero
eigenvalue of $P(\lambda)$. Then,
\[
\kappa_{\mathcal{T}_{P}}(\delta)=\kappa_{\mathcal{T}_{rev P}}\left(  \frac
{1}{\delta}\right)  .
\]

\end{lemma}

\begin{proof}
Let $\mathcal{T}_P(\lambda): = \lambda \mathcal{T}_1 - \mathcal{T}_0$. Then, from Lemma \ref{LTrevP}, 
\begin{align}\label{TrevP}
\mathcal{T}_{rev P}(\lambda) & = \lambda RD \mathcal{T}_{P}\left (\frac{1}{\lambda}\right ) DR = RD(\mathcal{T}_1 - \lambda \mathcal{T}_0) D R =: \lambda \tilde{\mathcal{T}}_1 - \tilde{\mathcal{T}}_0.
\end{align}

Moreover, if $y$ and $x$ are, respectively, a left and a right eigenvector of $\mathcal{T}_P(\lambda)$ associated with $\delta$, then $\tilde{y}=RDy$ and $\tilde{x}=RDx$ are, respectively,  a left and a right eigenvector  of $\mathcal{T}_{rev P}(\lambda)$ associated with $1/\delta$. Then, taking into account Theorem \ref{condform} and the fact that the Frobenius norm and  the spectral norm are unitarily invariant, we have 
\begin{align*} 
\kappa_{\mathcal{T}_{rev P}}\left (\frac{1}{\delta} \right) & = \frac{(\left | \frac{1}{\delta} \right | \|\tilde{\mathcal{T}}_1\|_2 + \|\tilde{\mathcal{T}}_0\|_2) \|\tilde{y}\|_2 \|\tilde{x}\|_2}{\left | \frac{1}{\delta} \right | | \tilde{y}^*\tilde{\mathcal{T}}_1 \tilde{x} |} \\
&=  \frac{( \|\tilde{\mathcal{T}}_1\|_2 + |\delta| \|\tilde{\mathcal{T}}_0\|_2) \|\tilde{y}\|_2 \|\tilde{x}\|_2}{ | \tilde{y}^*\tilde{T}_1 \tilde{x} |} = \frac{( \|\mathcal{T}_0\|_2 + |\delta| \|\mathcal{T}_1\|_2) \|y\|_2 \|x\|_2}{ | y^*\mathcal{T}_0 x |} \\
&= \frac{( \|\mathcal{T}_0\|_2 + |\delta| \|\mathcal{T}_1\|_2) \|y\|_2 \|x\|_2}{|\delta| | y^*\mathcal{T}_1 x |}= \kappa_{\mathcal{T}_P}(\delta),
\end{align*}
where the last equality follows from the fact that $(\delta \mathcal{T}_1 - \mathcal{T}_0)x =0.$
\end{proof}

\vspace{0.3cm}

\textbf{Proof of Theorem \ref{conditioningR}.}
Let $x_1$ and $x_2$ be, respectively, a left and a right eigenvector of $P(\lambda)$ associated with $\delta$.

We compare $\kappa_{\mathcal{T}_{P}}(\delta)$ with $\kappa_{P}(\delta)$.
Taking into account Lemmas \ref{lcondrev} and \ref{condPrevP}, we assume
$|\delta|\leq1,$ as otherwise we replace $P(\lambda)$ by $revP(\lambda)$ and
$\delta$ by $\frac{1}{\delta}.$ Note that both  $\rho_{1}$ and $\rho_{2}$  in (\ref{def-rho})
take the same value when considering $P(\lambda)$ and $revP(\lambda).$ From
Theorem \ref{condform} and Lemma \ref{kappatp}, we have
\begin{equation}\label{kappaT-P}
\frac{\kappa_{\mathcal{T}_P}(\delta)}{\kappa_{P}(\delta)}=\frac{(|\delta
|\Vert\mathcal{T}_{1}\Vert_{2}+\Vert\mathcal{T}_{0}\Vert_{2})\Vert
\Delta(\delta)x_{1}\Vert_{2}\Vert\Delta(\delta)x_{2}\Vert_{2}}{(\sum_{i=0}%
^{k}|\delta|^{i}\Vert A_{i}\Vert_{2})\Vert x_{1}\Vert_{2}\Vert x_{2}\Vert_{2}%
}.
\end{equation}
By Propositions \ref{prop} and  \ref{boundTp}, and taking into
account that each of the matrix coefficients of $\mathcal{T}_P(\lambda)$ contains an identity block, we get
\begin{align}\label{normT}
(|\delta|+1)\min\{\max\{1,\|A_k\|_2\}, \max\{1,\|A_0\|_2\} & \leq \nonumber \\
|\delta|\Vert
\mathcal{T}_{1}\Vert_{2} &+\Vert \mathcal{T}_{0}\Vert_{2}\leq 2 (|\delta|+1)\; \max_{i=0:k}\{1,\Vert
A_{i}\Vert_{2}\}. 
\end{align}
We also have 
 \begin{equation}\label{lowersumAi}
 \sum_{i=0}^k |\delta|^i \|A_i\|_2 \geq \|A_k\|_2 |\delta|^k + \|A_0\|_2 \geq \min\{\|A_k\|_2, \|A_0\|_2\} (|\delta|^k +1).
 \end{equation}
Thus, from (\ref{kappaT-P}) and the inequalities above, we get
\begin{equation}\label{quotientT}
\rho_2\frac{|\delta|+1}{\sum_{i=0}^k|\delta|^i} Q_T \leq \frac{\kappa_{\mathcal{T}_P}(\delta)}{\kappa_P(\delta)} \leq 2 \frac{\max_{i=0:k}\{1,\|A_i\|_2\}}{\min\{\|A_k\|_2, \|A_0\|_2\}} \frac{|\delta|+1}{|\delta|^k+1}Q_T,
\end{equation}
 where 
 \begin{equation}\label{QT}
 Q_T=\frac{\|\Delta(\delta) x_1\|_2 \|\Delta(\delta) x_2\|_2}{\|x_1\|_2\|x_2\|_2}\leq \|\Delta(\delta)\|_2^2\leq d_1(\delta) \max_{i=0:k} \{1, \|A_i\|_2^2\},
 \end{equation}
with $d_1(\delta)$ as in (\ref{Delta1}). Note that the last inequality follows from Lemma \ref{Deltax}.

 We now focus on the upper bound for $\frac{\kappa_{\mathcal{T}_P}(\delta)}{\kappa_P(\delta)}$. 
Taking into account Lemma \ref{d1bound}, and since $|\delta|\leq 1$, we obtain, for $k> 2$, 
\begin{align*}
\frac{|\delta|+1}{|\delta|^k+1} d_1(\delta) \leq 2 d_1(\delta) 
\leq  k^3.
\end{align*}

A better bound for $\frac{|\delta|+1}{|\delta|^k+1}d_1(\delta)$ can be obtained  when $|\delta| <<1$. More precisely, taking into account Lemma \ref{d1bound},  if 
$$ (k-1)^3|\delta|^{2} \leq k+1,$$
then,
$d_1(\delta) \leq k+1,$
implying that 
$$\frac{|\delta|+1}{|\delta|^k+1} d_1(\delta)
\leq 2(k+1).$$
Taking into account (\ref{quotientT}), (\ref{QT}) and the two previous upper bounds for $\frac{|\delta|+1}{|\delta|^k+1} d_1(\delta) $, the upper bound part of the theorem follows.

Next we show the  lower bound for $\frac{\kappa_{\mathcal{T}_P}(\delta)}{\kappa_P(\delta)} $. From  (\ref{first-line}) with $x$ replaced by $x_1$, we have
$$\|\Delta(\delta) x_1 \|_2^2 \geq  \|x_1\|_2^2\sum_{r=0}^{\frac{k-1}{2}}|\delta|^{2r}. $$
An analogous inequality holds for $x_2$. Thus,
\begin{align*}
\frac{|\delta|+1}{\sum_{i=0}^k |\delta|^i} Q_T & \geq \frac{(|\delta|+1) \sum_{r=0}^{\frac{k-1}{2}} |\delta|^{2r}}{
\sum_{i=0}^k |\delta|^i} =1,
\end{align*}
and the result follows from (\ref{quotientT}).

\subsection{Proof of Theorem \ref{backwardR}} \label{sec;backward}

The following lemma will allow us to only consider eigenvalues $\delta$ such that $|\delta| \leq 1$ when proving Theorem \ref{backwardR}.

\begin{lemma}
\label{backPrevP}Let $P(\lambda)$ be a regular matrix polynomial of odd degree
$k$ as in (\ref{pol}) with $A_0 \neq 0$, and let $(z,\delta)$ be an approximate right eigenpair
of $\mathcal{T}_{P}(\lambda)$, with $\delta \neq 0$. Then, $(RDz, \frac{1}{\delta})$ is an approximate right eigenpair of $\mathcal{T}_{rev P}(\lambda)$ and 
\[
\eta_{T_{P}}(z,\delta)=\eta_{\mathcal{T}_{rev P}}\left( R D z,\frac{1}{\delta
}\right), 
\]
where $R$ and $D$ are as in (\ref{flipR}) and (\ref{DD}), respectively.
\end{lemma}

\begin{proof}
 Let $\mathcal{T}_P(\lambda): = \lambda \mathcal{T}_1 - \mathcal{T}_0$ and let $\mathcal{T}_{rev P}(\lambda):=\lambda \tilde{\mathcal{T}}_1 - \tilde{\mathcal{T}}_0.$ Note that (\ref{TrevP}) holds.  Then, from Theorem \ref{back} and Lemma \ref{LTrevP}, 


\begin{align*} 
\eta_{\mathcal{T}_{rev P}}\left (R D z, \frac{1}{\delta} \right) &= \frac{\| \mathcal{T}_{rev P}(\frac{1}{\delta}) RD z\|_2}{(\left |\frac{1}{\delta} \right |\|\tilde{\mathcal{T}}_1\|_2 +\|\tilde{\mathcal{T}}_0\|_2) \|R D z\|_2} = \frac{\| \frac{1}{\delta} \mathcal{T}_{ P}(\delta) z\|_2}{(\left |\frac{1}{\delta}\right | \|\tilde{\mathcal{T}}_1\|_2 +\|\tilde{\mathcal{T}}_0\|_2) \|z\|_2} \\
& = \frac{\| \mathcal{T}_{ P}(\delta) z\|_2}{ (\|\mathcal{T}_0\|_2 + |\delta| \|\mathcal{T}_1\|_2) \|z\|_2} = \eta_{\mathcal{T}_P}(z, \delta).\\
\end{align*}
\end{proof}

\vspace{0.3cm}

\textbf{Proof of Theorem \ref{backwardR}.}
Let $(z, \delta)$ be an approximate right eigenpair of $\mathcal{T}_P(\lambda):=\lambda \mathcal{T}_1 - \mathcal{T}_0$ and assume that $|\delta| \leq 1$.
Let $x:=(e_{k}^T \otimes I_n) z$. Note that $(x, \delta)$ can be seen as an approximate right eigenpair of $P(\lambda)$. First we show the upper bound in (\ref{back-bound}). 
We have
$$P(\delta)x = P(\delta) (e_k^T \otimes I_n)z =( e_k^T \otimes P(\delta))z= \Delta^{\mathcal{B}}(\delta) \mathcal{T}_P(\delta)z,$$
where the last equality follows from Lemma \ref{GFP-ans}. Thus,
\begin{align}
\frac{\eta_P( x, \delta)}{\eta_{\mathcal{T}_P}(z, \delta)}& =\frac{\norm{P(\delta)x}_2}{(\sum_{i=0}^{k} |\delta|^{i}\norm{A_i}_2)\norm{x}_2} \cdot \frac{(|\delta| \norm{\mathcal{T}_1}_2+\norm{\mathcal{T}_0}_2)\norm{z}_2}{\norm{\mathcal{T}_P(\delta)z}_2} \nonumber \\
&\leq \frac{\|\Delta^{\mathcal{B}}(\delta)\|_2 \|\mathcal{T}_P(\delta)z\|_2}{(\sum_{i=0}^k |\delta|^i \|A_i\|_2)\|x\|_2}\cdot \frac{(|\delta| \norm{\mathcal{T}_1}_2+\norm{\mathcal{T}_0}_2)\norm{z}_2}{\norm{\mathcal{T}_P(\delta)z}_2} \nonumber \\
& = \frac{\|\Delta(\delta)\|_2 (|\delta| \|\mathcal{T}_1\|_2 +\|\mathcal{T}_2\|_2)}{\sum_{i=0}^k |\delta|^i \|A_i\|_2} \cdot \frac{\|z\|_2}{\|x\|_2}\nonumber\\
& \leq 2 \frac{\max_{i=0:k} \{1,\|A_i\|_2\}}{\min \{\|A_k\|_2, \|A_0\|_2\}} \frac{|\delta| +1}{|\delta|^k +1} \frac{\|z\|_2}{\|x\|_2} \|\Delta(\delta)\|_2, \label{supercota}
\end{align}
where the last inequality follows from (\ref{lowersumAi}) and  the second inequality in (\ref{normT}). 

By Lemma \ref{Deltax}, we have
$$\|\Delta(\delta)\|_2\leq \sqrt{d_1(\delta)} \max_{i=0:k} \{1, \|A_i\|_2\},$$
where $d_1(\lambda)$ is as in (\ref{Delta1}).  Taking into account Lemma
\ref{d1bound}, since $|\delta|\leq 1$, we have, for $k> 2$, 
$$\frac{|\delta|+1}{|\delta|^k+1} \sqrt{d_1(\delta)} \leq \frac{|\delta|+1}{|\delta|^k+1} \sqrt{\frac{k+1}{2} +  \frac{(k-1)^3}{2}|\delta|^{2}} \leq 2k^{3/2}.$$
Thus, the upper bound in (\ref{back-bound}) follows by combining the two previous bounds with (\ref{supercota}). 

If $|\delta| \ll1$, a better bound can be obtained. More precisely, if 
$$  (k-1)^3 |\delta|^{2} \leq  k+1,$$
then
$$\frac{|\delta|+1}{|\delta|^k+1} \sqrt{d_1(\delta)} \leq  2\sqrt{k+1},$$ 
implying the upper bound for $\frac{\eta_P(x, \delta)}{\eta_{\mathcal{T}_P(z, \delta)}}$ in (\ref{second-back-T}). 

Now suppose that $|\delta|>1$. From Lemmas \ref{lbackrev}
and \ref{backPrevP}, we have
\[
\frac{\eta_{P}((e_{1}^{T}\otimes I_{n})z,\delta)}{\eta_{\mathcal{T}_{P}%
}(z,\delta)}=\frac{\eta_{revP}((e_{k}^{T}\otimes I_{n})RDz,\frac{1}{\delta}%
)}{\eta_{\mathcal{T}_{revP}}(RDz,\frac{1}{\delta})}.
\]
Taking into account Lemma \ref{LTrevP} and since $|\frac{1}{\delta}|<1,$ by
the part of the theorem already proved, we have that (\ref{back-bound}) and
(\ref{second-back-T}) hold with $\frac{\eta_{P}(x,\delta)}{\eta_{T_{P}%
}(z,\delta)}$ replaced by
\[
\frac{\eta_{revP}((e_{k}^{T}\otimes I_{n})RDz,\frac{1}{\delta})}%
{\eta_{\mathcal{T}_{revP}}(RDz,\frac{1}{\delta})}.
\]
Note that%
\[
\frac{\Vert z\Vert_{2}}{\Vert(e_{1}^{T}\otimes I_{n})z\Vert_{2}}=\frac{\Vert
RDz\Vert_{2}}{\Vert(e_{k}^{T}\otimes I_{n})RDz\Vert_{2}}.
\]

\section{Conditioning and backward error of  $D_1(\lambda, P)$, $D_k(\lambda,P)$ and $C_1(\lambda)$}\label{cond-D1Dk}

 In this section we present results analogous to Theorems \ref{conditioningR} and \ref{backwardR} for the linearizations $D_1(\lambda, P)$, $D_k(\lambda, P)$ and $C_1(\lambda)$. We note that these results were previously obtained in \cite{backward} and \cite{tisseur}. We  include them for completeness with the goal of comparing the conditioning and backward error of $\mathcal{T}_P(\lambda)$ and those of the linearizations $D_1(\lambda, P)$, $D_k(\lambda, P)$, and $C_1(\lambda)$. With respect to the conditioning, we give some improvements in the bounds, which allow us to obtain a more accurate comparison  of the different linearizations.

We start by recalling the definition of the pencils that we are considering in this section.  The reader can find more details in \cite{HMMT, 4m-vspace}.

Let $P(\lambda)$ be a matrix polynomial of degree $k$ as in (\ref{pol}) and assume that $k\geq 2$. We have
$$D_1(\lambda, P) :=\lambda  \left [ \begin{array}{c|cccc} A_k &  & &  &\\ \hline  & -A_{k-2} & -A_{k-3} & \cdots & -A_0 \\  & -A_{k-3} & -A_{k-4} & \cdots & 0 \\ & \vdots & \iddots & & \vdots \\ & - A_0 & 0  & \cdots & 0\end{array} \right]- \left[ \begin{array}{ccccc} 
-A_{k-1} & -A_{k-2} & \cdots & -A_1 & -A_0 \\ -A_{k-2} & -A_{k-3} & \cdots & -A_0 & 0 \\ \vdots &  \iddots & & &\vdots  \\ -A_0 & 0 & \cdots & \cdots & 0\end{array} \right], $$
$$D_k(\lambda, P) :=\lambda  \left [ \begin{array}{cccc} 0 & \cdots   & 0  & A_k \\
0 &\cdots &  A_k & A_{k-1} \\ \vdots & \iddots &\vdots & \vdots  \\ A_k  & \cdots & A_2 & A_1 \end{array} \right] - \left[ \begin{array}{cccc|c}  0 & \cdots & 0 & A_k & \\ 0 & \cdots & A_k & A_{k-1}\\ \vdots & \iddots & & \vdots & \\ A_k & \cdots & A_3 & A_2 & \\
\hline
&&&& -A_0 \end{array}\right], $$
$$C_1(\lambda):= \lambda \left[ \begin{array}{cccc}  A_k &&& \\ & I_n &&\\ & & \ddots & \\ &&& I_n \end{array}\right] - 
\left[ \begin{array}{cccc} -A_{k-1} & -A_{k-2} & \cdots & -A_0 \\ I_n & 0 & \cdots & 0 \\ \vdots & \ddots & \ddots & \vdots \\ 
0 & \cdots & I_n & 0\end{array}\right].$$

We emphasize that $C_1(\lambda)$ is the very well-known first Frobenius companion form, which is
fundamental in the theory and in the numerical computations of matrix polynomials \cite{GLR-book2}. The block-symmetric pencils $D_1(\lambda, P)$ and $D_k(\lambda, P)$ have been thoroughly studied recently in \cite{SCDLP, backward, tisseur, HMMT, 4m-vspace, goodvibrations}, although they were introduced as early as in \cite{lancaster}.  Recall that $D_1(\lambda, P)$, $D_k(\lambda, P) \in \mathbb{DL}(P)$, where $\mathbb{DL}(P)$ denotes the vector space of pencils defined in \cite{4m-vspace}.

When comparing the conditioning of eigenvalues and the backward error of approximate eigenpairs of  a regular $P(\lambda)$ with those of any of its linearizations $D_1(\lambda,P)$, $D_k(\lambda,P)$ and $C_1(\lambda)$, the next two lemmas will be useful since they provide a way to recover an eigenvector of a matrix polynomial $P(\lambda)$ from an eigenvector of $D_1(\lambda,P)$, $D_k(\lambda, P)$ and $C_1(\lambda)$.  We will use the following notation:
\begin{equation}\label{delta}
\Lambda(\lambda)=\left[  \lambda
^{k-1}\text{ }\cdots\lambda\text{ }1\right]  ^{T}.
\end{equation}

\begin{lemma}\label{eigD1Dk}\cite[Theorem 3.8]{4m-vspace}
Let $P(\lambda)$ be an $n\times n$  regular matrix polynomial of degree $k$ and  $\delta$ be a finite eigenvalue of $P(\lambda)$. Let $L(\lambda) \in \mathbb{DL}(P)$ be a linearization of $P(\lambda)$. A vector $z$ is a right eigenvector of $L(\lambda)$ associated with  $\delta$ if and only if $z=\Lambda(\delta) \otimes x$ for some  right eigenvector $x$ of $P(\lambda)$ associated with $\delta$. Similarly, a vector $\omega$ is a left eigenvector of $L(\lambda)$ associated with $\delta$ if and only if $\omega=\overline{\Lambda(\delta)} \otimes y$  for  some left eigenvector $y$  of $P(\lambda)$ associated with $\delta$.  
\end{lemma}

\begin{lemma}\label{eigC1}\cite[Section 1 and Lemma 7.2]{tisseur}
Let $P(\lambda)$ be an $n\times n$ regular matrix polynomial of degree $k$ and $\delta$ be a finite eigenvalue of $P(\lambda)$.  A vector $z$ is a right eigenvector of $C_1(\lambda)$ associated with  $\delta$ if and only if $z=\Lambda(\delta) \otimes x$ for some right eigenvector $x$ of $P(\lambda)$ associated with $\delta$. 

A vector $\omega$ is a left eigenvector of $C_1(\lambda)$ associated with  $\delta$ if and only if $\omega^*= y^*[I_n, P_1(\delta), \ldots, P_{k-2}(\delta), \allowbreak P_{k-1}(\delta)]$,  for some left eigenvector $y$ of $P(\lambda)$ associated with $\delta$, where $P_i$ is as in (\ref{pol-Pi}).
Thus, if $\omega$ is a left eigenvector of $C_1(\lambda)$ with eigenvalue $\delta$, then $y=(e_1^T \otimes I_n) \omega$ is a left eigenvector of $P(\lambda)$ with eigenvalue $\delta$. Moreover, any left eigenvector $y$  of $P(\lambda)$ with eigenvalue $\delta$ can be recovered from some left eigenvector $\omega$ of  $C_1(\lambda)$ by taking $y= (e_1^T\otimes I_n)\omega$. 
\end{lemma}

\subsection{Conditioning and backward error of $D_1(\lambda,P)$ and $D_k(\lambda,P)$}

Next we recall some well-known results on the conditioning of eigenvalues and backward error of approximate eigenpairs of
the linearizations $D_1(\lambda, P)$ and $D_k(\lambda, P)$  of a regular  matrix polynomial  $P(\lambda)$ introduced previously in \cite{backward, tisseur}.  Moreover, we sharpen one of the results in \cite{tisseur}. Our goal is to compare the results in  Section \ref{main}  for $\mathcal{T}_P(\lambda)$ and $\mathcal{R}_P(\lambda)$  with those for $D_1(\lambda, P)$ and $D_k(\lambda, P)$. 

Recall that $D_1(\lambda, P)$ (resp. $D_k(\lambda, P)$) is a linearization of a regular matrix polynomial  $P(\lambda)$ as in (\ref{pol}) if and only if $A_0$ (resp. $A_k$) is nonsingular. 

In \cite{tisseur}, it was shown that, for a simple, finite, nonzero eigenvalue $\delta$ of a regular  $P(\lambda)$ as in (\ref{pol}) with $A_0\neq 0$,  the condition number of $\delta$ as an eigenvalue  of $D_1(\lambda, P)$ (resp. $D_k(\lambda, P)$), when $A_0$ (resp. $A_k$)  is nonsingular, is close to optimal among the linearizations of   $P(\lambda)$  in $\mathbb{DL}(P)$, when $|\delta| \geq 1$ (resp.  $|\delta | \leq 1$), provided that 
\begin{equation}
\rho:=\frac{\max_{i=0:k}\{\| A_{i}\|_{2}\}}{\min \{\Vert A_{k}\Vert_{2},\Vert
A_{0}\Vert_{2}\}}. \label{rhoD}%
\end{equation}
is of order 1, which  implies, in particular, that  all matrix coefficients of $P(\lambda)$ must have similar norm. 
Notice  that  $\rho \geq 1$.

The combination of Theorems 4.4 and 4.5 in \cite{tisseur} provides a lower and an upper bound for  $\frac{\kappa_{D_t}(\delta)}{\kappa_P(\delta)}$,  when either $t=1$ and $|\delta| \geq 1$, or $t=k$ and $|\delta| \leq 1$, namely,
$$\left (\frac{2\sqrt{k}}{k+1}\right) \frac{1}{\rho}\leq\frac{\kappa_{D_{t}}(\delta)}
{\kappa_{P}(\delta)}\leq\sqrt{k^{7}}\rho^{2}.$$

Using the same techniques as those applied to compute the bounds in Theorem \ref{conditioningR}, next we deduce sharper bounds  for the quotient $\frac{\kappa_{D_t}(\delta)}{\kappa_P(\delta)}$ than those provided in \cite{tisseur}. Based on the results obtained,  we can provide a fair comparison of the linearizations   $\mathcal{T}_P(\lambda)$, $D_1(\lambda, P)$ and $D_k(\lambda, P)$ with respect to conditioning and explain the numerical experiments in Section \ref{numerical} appropriately.


\begin{theorem}\label{Di-cond}
Let $P(\lambda)$ be a regular matrix polynomial of degree $k$ as in (\ref{pol}) with $A_0\neq 0$. Assume that  $\delta$ is a simple, finite, nonzero eigenvalue of $P(\lambda)$. Let  $t\in \{1, k\}$ and suppose that $A_0$ is nonsingular if $t=1$, and $A_k$ is nonsingular if $t=k$. Let  $\rho$ be as in (\ref{rhoD}). 
Then, if either  $t=1$ and $|\delta| \geq 1$, or $t=k$ and $|\delta| \leq  1$, we have 
$$\frac{1}{\rho} \leq \frac{\kappa_{D_t}(\delta)}{\kappa_{P}(\delta)}\leq k^2 \rho .$$
 \end{theorem}
 
 \begin{proof}
 Let $x$ and $y$ be a right and a left eigenvector of $P(\lambda)$
associated with $\delta$, respectively. Let  $\Lambda(\lambda)$ be as in (\ref{delta}).
 Let  $t\in\{1,k\},$  and define $D_{t}
(\lambda,P) : =L_{1}^t\lambda-L_{0}^t$. Taking into account Theorem 3.2 in 
\cite{tisseur}, and using the natural weights for $\kappa_{D_t}(\delta)$, we have
\begin{equation}
\frac{\kappa_{D_{t}}(\delta)}{\kappa_{P}(\delta)}=\frac{(|\delta|\Vert
L_{1}^t\Vert_{2}+\Vert L_{0}^t\Vert_{2})\Vert\Lambda(\delta)\Vert_{2}^{2}}%
{|\delta|^{k-t}\sum_{i=0}^{k}|\delta|^{i}\Vert A_{i}\Vert_{2}}. \label{quotDi}%
\end{equation}
By Proposition \ref{prop} and taking into account that all block-entries of
$L_{1}^t$ and $L_{0}^t$ are either $0$ or matrix coefficients of $P(\lambda)$, we
obtain
\begin{equation}\label{boundsL1L0}
  |\delta| \|L_1^t\|_2 + \|L_0^t\|_2 \leq (|\delta|+1) k\; \max_{i=0:k}\{\|A_i\|_2\}.
 \end{equation}
Thus, from (\ref{lowersumAi}) and  (\ref{quotDi}), we obtain 
\begin{align}\label{quotient1}
\frac{\kappa_{D_{t}}(\delta)}{\kappa_{P}(\delta)}&\leq \frac{k\; \max_{i=0:k} \{
\|A_{i}\|_2 \}}{\min \{\| A_{k}\|_2,\Vert A_{0}\Vert_{2}\}}\frac
{( |\delta |+1)\sum_{i=0}^{k-1}|\delta|^{2i}}{(|\delta|^{k}+1)|\delta|^{k-t}}
= k \rho \frac
{\sum_{i=0}^{2k-1}|\delta|^{i}}{(|\delta|^{k}+1)|\delta|^{k-t}}.
\end{align}
Clearly, if $|\delta|=1,$ the upper bound in the statement follows. 
Now
suppose that $t=1$ and $|\delta | >1$, or $t=k$ and $|\delta|<1$.  Since $|\delta|^{2k}-1= (|\delta|-1) \sum_{i=0}^{2k-1} |\delta|^i$, from (\ref{quotient1}), we have
\begin{align*}
\frac{\kappa_{D_{t}}(\delta)}{\kappa_{P}(\delta)}  & \leq k\rho\frac
{|\delta|^{2k}-1}{(|\delta|-1)(|\delta|^{k}+1)|\delta|^{k-t}}
=k\rho\frac{|\delta|^{k}-1}{(|\delta|-1)|\delta|^{k-t}}\\
& \leq k\rho\frac{|\delta|^{k-1}+\cdots+1}{|\delta|^{k-t}}\leq k^{2}\rho
\frac{\max\{1,|\delta|^{k-1}\}}{|\delta|^{k-t}}= k^{2}\rho.
\end{align*}
Now we show the lower bound in the statement. From Proposition \ref{prop} and taking into account the block-structure of $D_t(\lambda, P)$, we obtain
$$(|\delta|+1) \min \{\|A_k\|_2, \|A_0\|_2\} \leq
  |\delta| \|L_1^t\|_2 + \|L_0^t\|_2. $$
Thus, from  (\ref{quotDi}),
\begin{align*}
\frac{\kappa_{D_{t}}(\delta)}{\kappa_{P}(\delta)}  &  \geq\frac{(|\delta
|+1)\min \{\Vert A_{k}\Vert_{2},\Vert A_{0}\Vert_{2}\}\sum_{i=0}^{k-1}%
|\delta|^{2i}}{\max_{i=0:k} \{\|A_{i}\|_{2}\}|\delta|^{k-t}\sum_{i=0}^{k}%
|\delta|^{i}}\\
&  =\frac{1}{\rho}\frac{\sum_{i=0}^{2k-1}|\delta|^{i}}{|\delta|^{k-t}%
\sum_{i=0}^{k}|\delta|^{i}}\geq\frac{1}{\rho}.
\end{align*}
 \end{proof}

 Notice that the previous theorem implies that,  if the norms of the matrix coefficients of $P(\lambda)$ have similar magnitudes, that is, $\rho \approx 1$, and $k$ is moderate, then the conditioning of  the eigenvalues of $D_1(\lambda, P)$ (resp. $D_k(\lambda, P)$)  is close to that of the corresponding eigenvalues of  $P(\lambda)$ when $|\delta|\geq 1$ and  $A_0$ is nonsingular (resp. when $|\delta|\leq 1$ and $A_k$ is nonsingular). 
 
 \begin{remark}\label{remarkD1}
Based on Theorems \ref{conditioningR} and  \ref{Di-cond}, we  compare the conditioning  of the nonzero finite simple eigenvalues of the linearizations $\mathcal{T}_P(\lambda)$, $D_1(\lambda, P)$ and $D_k(\lambda, P)$. We note first that the bounds on the quotient of condition numbers associated with $\mathcal{T}_P(\lambda)$ are valid for all eigenvalues, regardless of their modulus, while the bound for $D_1(\lambda, P)$ (resp. $D_k(\lambda, P)$) is only valid in a certain range of eigenvalues. This forces the use of both linearizations, $D_1(\lambda, P)$ and $D_k(\lambda, P)$,  when the matrix polynomial has both eigenvalues with modulus larger than 1 and eigenvalues with modulus less than 1. We note also that, when $P(\lambda)$ is scaled by dividing all matrix coefficients by $\max_{i=0:k} \{\|A_i\|_2\}$, the parameters $\rho$ and $\rho_1$ that appear in the bounds of the quotients of condition numbers are equal. We must also point out that, although the general upper bound for these quotients has degree 2 on $k$ for $D_1(\lambda, P)$ and $D_k(\lambda, P)$, and degree 3 for $\mathcal{T}_P(\lambda)$, when the modulus of the eigenvalue is not close to 1, the bound for $\mathcal{T}_P(\lambda)$ has degree 1 on $k$. Finally, observe that  $D_1(\lambda,P)$ (resp. $D_k(\lambda, P)$) is a linearization of $P(\lambda)$ if and only if  $A_0$ (resp. $A_k$) is nonsingular, contrarily to  $\mathcal{T}_P(\lambda)$, which is always a linearization of  $P(\lambda)$. Thus, it is clear that the use of $\mathcal{T}_P(\lambda)$ presents clear advantages over the combined use of $D_1(\lambda, P)$ and $D_k(\lambda, P)$ for conditioning purposes.
 \end{remark}

 We now recall a result that compares the backward error of an approximate eigenpair of the  linearization $D_t(\lambda, P)$, $t\in \{1, k\}$, 
  with the backward error of a certain approximate eigenpair of $P(\lambda)$.   

 \begin{theorem}\cite[Corollary 3.11]{backward}\label{backwardD}
 Let $P(\lambda)$ be a  matrix polynomial of degree $k$ as in (\ref{pol}) with $A_0 \neq 0$. Let $t\in \{1, k\}$ and suppose that $A_0$ is nonsingular if $t=1$, and $A_k$ is nonsingular if $t=k$. Let $\rho$ be as in (\ref{rhoD}). Let $(z, \delta)$ be an approximate right eigenpair of $D_t(\lambda, P)$, with $\delta$ nonzero and finite. Then, for $z_t = (e_t^T \otimes I_n)z$,  we have that $(z_t, \delta)$ is an approximate right eigenpair for $P(\lambda)$ and 
 \begin{equation}\label{backDt}
  \frac{\eta_P(z_t, \delta)}{\eta_{ D_t}(z, \delta)} \leq k^{3/2} \frac{\|z\|_2}{\|z_t\|_2}\rho.\end{equation}
 An analogue result holds for left eigenpairs $(w^*, \delta)$ of $D_t(\lambda, P)$ by simply replacing $z$ by $w$ and $z_t$ by  $w_t:= (e_t^T \otimes I_n)w$. 
 \end{theorem}
 
Taking into account the form of the right eigenvectors of $D_1(\lambda,P)$ and $D_k(\lambda,P)$ (see Lemma \ref{eigD1Dk}), and assuming that the approximate eigenvector $z$ has a similar
block-structure, it is expected  that, if $t=1$ and $|\delta| \geq 1$, or if $t=k$ and $|\delta| \leq 1$, the approximate eigenvector $z_t$ for $P(\lambda)$,  recovered from the approximate eigenvector $z$ of $D_t(\lambda, P)$ as in Theorem \ref{backwardD}, makes the   quotient $\frac{\|z\|_2}{\|z_t\|_2}$ in (\ref{backDt})  close to 1.  

  \begin{remark}\label{remarkD2}
 Based on Theorems \ref{backwardR} and \ref{backwardD}, we compare the backward errors of  $\mathcal{T}_P(\lambda)$, $D_1(\lambda, P)$ and $D_k(\lambda, P)$.  First note that, although Theorem \ref{backwardD} is valid for any value of $\delta$, an argument similar to the one in Remark \ref{quotientzx} shows that, for $D_1(\lambda, P)$, the quotient $\frac{\|z\|_2}{\|z_1\|_2}$ in (\ref{backDt}) is expected to be close to one only if $|\delta| \geq 1$, while for $D_k(\lambda, P)$, the quotient $\frac{\|z\|_2}{\|z_k\|_2}$  is expected to be close to 1 only if $|\delta| \leq 1$. In contrast, according to Remark \ref{quotientzx}, the quotient $\frac{\|z\|_2}{\|x\|_2}$ appearing in Theorem \ref{backwardR} is expected to be close to one for any value of $\delta$. In addition, observe that, when  $P(\lambda)$ is scaled by dividing all the matrix coefficients by $\max_{i=0:k} \{\|A_i\|_2\}$, the parameters $\rho$ and $\rho'$ that appear in the bounds in Theorem \ref{backwardR} and \ref{backwardD}, respectively, have the same value, which is approximately one for polynomials whose coefficients have similar norms. Finally, observe that the bounds in Theorem \ref{backwardR} and \ref{backwardD} have the same dependence on $k$, that is, $k^{3/2}$, in general, and that the bound corresponding to $\mathcal{T}_p(\lambda)$ is improved  if $|\delta|$ is not close to 1, depending on  $k^{1/2}$ in this case. Therefore, once $P(\lambda)$ is divided by $\max_{i=0:k} \{ \|A_i\|_2\}$, the use of $\mathcal{T}_p(\lambda)$ presents clear advantages with respect to  $D_1(\lambda, P)$ and $D_k(\lambda, P)$ in terms of backward errors, since by using $\mathcal{T}_p(\lambda)$, for all approximate eigenvalues,  we will get similar backward errors as using $D_1(\lambda, P)$ for computing the eigenvalues with $|\delta| \geq 1$ and $D_k(\lambda, P)$, for computing the eigenvalues with $|\delta| \leq 1$.

  \end{remark}

 \subsection{Conditioning and backward error of $C_1(\lambda)$}
Next we focus on the first Frobenius companion linearization $C_{1}(\lambda)$ of
$P(\lambda).$   Note that,  since $C_2(\lambda)=[C_1(\lambda)]^{\mathcal{B}}$, where $C_1(\lambda)^{\mathcal{B}}$ denotes the block-transpose of $C_1(\lambda)$ and $C_2(\lambda)$ denotes the second Frobenius companion form of $P(\lambda)$ \cite{GLR-book2},   any result that we produce for $C_1(\lambda)$ has an immediate counterpart for $C_2(\lambda)$ (see  \cite[Lemma 7.1]{tisseur}).


We start by comparing the conditioning of the eigenvalues of $C_{1}(\lambda)$ with the
conditioning of the corresponding eigenvalues of $P(\lambda)$. As far as we know, an explicit result, valid
for any $k,$ is not given in the literature, though the quadratic case was
studied in \cite{tisseur}. 


 \begin{theorem}\label{conditioningC} Let $P(\lambda)$ be a regular matrix polynomial of degree $k$ as in (\ref{pol}) with $A_0\neq 0$.  
 Let $\delta$ be a simple, finite, nonzero eigenvalue of $P(\lambda)$. Let $C_1(\lambda)$ be the first  Frobenius companion linearization of $P(\lambda)$. Let 
 \begin{equation}\label{rhonu}
 \rho':=\frac{\max_{i=0:k} \{1,  \|A_i\|^2_2\} }{\min\{\|A_k\|_2, \|A_0\|_2\}}\quad \textrm{and} \quad \nu := \frac{\min\{ \max\{1, \|A_k\|_2\}, \max_{i=0:k-1}\{1,\|A_i\|_2\}\} }{\max_{i=0:k} \{ \|A_i\|_2\}}.
 \end{equation}
 Then, 
 $$\frac{\nu}{k+1} \leq \frac{\kappa_{C_1}(\delta)}{\kappa_P(\delta)} \leq   2\sqrt{2}k^3 \rho' .$$
Moreover, if $|\delta|\geq \sqrt{(k-1)^3}$ or $|\delta| \leq \frac{1}{2}$, then 
$$\frac{\kappa_{C_1}(\delta)}{\kappa_P(\delta)} \leq  \frac{4}{3} k(1+k)\rho'.$$

 \end{theorem}

 \begin{proof}
Let $C_1(\lambda):= \lambda X_1 + Y_1$.  By Lemma \ref{eigC1}, 
if  $y$ is a left-eigenvector of
$P(\lambda)$ associated with $\delta$, then
  $$w= \left [ \begin{array}{c} I_n \\ P_1^*(\delta) \\ \vdots \\ P_{k-1}^* (\delta) \end{array} \right] y, $$ 
 is a left eigenvector of $C_1(\lambda)$ associated with $\delta$. Taking into account \cite[Theorem 7.3]{tisseur}, we obtain%
\[
\frac{\kappa_{C_{1}}(\delta)}{\kappa_{P}(\delta)}=\frac{\Vert w\Vert_{2}%
}{\Vert y\Vert_{2}}\frac{(|\delta|\Vert X_{1}\Vert_{2}+\Vert Y_{1}\Vert
_{2})\Vert\Lambda(\delta)\Vert_{2}}{\sum_{i=0}^{k}|\delta|^{i}\Vert A_{i}\Vert_{2}},
\]
where $\Lambda(\lambda)$ is as in (\ref{delta}). By \cite[Theorem
7.4]{tisseur}, we have
\begin{equation}
\Vert X_{1}\Vert_{2}=\max\{1,\Vert A_{k}\Vert_{2}\}\label{f1}%
\end{equation}
and
\begin{equation}
\max_{i=0:k-1}\{1,\Vert A_{i}\Vert_{2}\}\leq\Vert Y_{1}\Vert_{2}\leq
k\;\max_{i=0:k-1}\{1,\Vert A_{i}\Vert_{2}\}.\label{f2}%
\end{equation}
Thus, using (\ref{lowersumAi}), we get
\begin{equation}
\frac{\kappa_{C_{1}}(\delta)}{\kappa_{P}(\delta)}\leq\frac{\Vert w\Vert_{2}%
}{\Vert y\Vert_{2}}\frac{(|\delta|+k)\max_{i=0:k}\{1,\Vert A_{i}\Vert_{2}%
\}}{\min\{\Vert A_{k}\Vert_{2},\Vert A_{0}\Vert_{2}\}
}\frac{\sqrt{\sum_{i=0}^{k-1}|\delta|^{2i}}}{|\delta|^{k}+1}.\label{f3}%
\end{equation}

Assume that $|\delta|\leq1$. Then, 
\begin{align}
\| w\|_{2}^{2} &  =\Vert y\Vert_{2}^{2}+\sum_{i=1}^{k-1}\Vert
P_{i}^{\ast}(\delta)y\Vert_{2}^{2}\leq\left(  1+\sum_{i=1}%
^{k-1}\Vert P_{i}^{\ast}(\delta)\Vert_{2}^{2}\right)  \Vert
y\Vert_{2}^{2}\nonumber\\
&  \leq\;\max_{i=0:k}\{1,\Vert A_{i}\Vert_{2}^2\}\Vert y\Vert_{2}^{2}\left(
1+\sum_{i=1}^{k-1}\left(  \sum_{j=0}^{i}|\delta|^{j}\right)  ^{2}\right)
\nonumber\\
&  \leq\;\max_{i=0:k}\{1,\Vert A_{i}\Vert_{2}^2\}\Vert y\Vert_{2}%
^{2}\left(  1+(k-1)k\sum_{j=0}^{k-1}|\delta|^{2j}\right)
,\label{w3}%
\end{align}
where the second and third inequalities follow taking into account Lemma
\ref{Pjibound} and Lemma \ref{tech1}, respectively. Thus,
\begin{align*}
\Vert w\Vert_{2}^{2} &  \leq \max_{i=0:k}\{1,\Vert A_{i}\Vert_{2}^2\}\Vert
y\Vert_{2}^{2}\left[  1+(k-1)k^2\right]    \leq k^{3}\max_{i=0:k}\{1,\Vert A_{i}\Vert_{2}^2\}\Vert y\Vert_{2}%
^{2}.
\end{align*}
From (\ref{f3}), we obtain
$
\frac{\kappa_{C_{1}}(\delta)}{\kappa_{P}(\delta)}\leq k^{2}
(1+k)\rho^{\prime}\leq 2\sqrt{2} k^3 \rho',
$
when $|\delta|\leq 1$.

A better bound for $\kappa_{C_1}(\delta)/\kappa_P(\delta)$ can be found when $|\delta| \ll 1$.  Assume that $|\delta|\leq \frac{1}{2}$.  Taking into account (\ref{f3}) and the upper bound for $\|w\|_2$ in (\ref{w3}),  we have
\begin{align*}
\frac{\kappa_{C_{1}}(\delta)}{\kappa_{P}(\delta)} &  \leq (1+k) \rho' \frac{\sqrt{(1+k(k-1)\sum_{j=0}^{k-1} |\delta|^{2j})\sum_{i=0}^{k-1} |\delta|^{2i}}}{|\delta|^k +1}\\
& \leq (1+k)k \rho' \frac{\sqrt{\sum_{j=0}^{k-1} |\delta|^{2j} \sum_{i=0}^{k-1}|\delta|^{2i}}}{|\delta|^k+1}\\
&\leq (1+k)k \rho' \frac{\sum_{j=0}^{k-1} |\delta|^{2j}}{|\delta|^k+1} =(1+k)k \rho' \frac{(1-|\delta|^{2k})}{(1-|\delta|^2)(1+|\delta|^k) } \\
& \leq  (1+k)k\rho' \frac{1} { (1-|\delta|^2)}  \leq \frac{4}{3} (1+k)k \rho',
\end{align*}
where the second inequality follows  because
$$1+k(k-1) \sum_{j=0}^{k-1}|\delta|^{2j}\leq \sum_{j=0}^{k-1}|\delta|^{2j}+(k^2-k) \sum_{j=0}^{k-1}|\delta|^{2j} \leq k^2 \sum_{j=0}^{k-1} |\delta|^{2j},$$
and the last inequality follows from the assumption $|\delta|\leq
\frac{1}{2}.$

If $|\delta|>1,$ taking into account  Lemma \ref{tech1}, Lemma \ref{PrPr}, and (\ref{Pjibound}), we have 
\begin{align}
\Vert w\Vert_{2}^{2}  &  =\Vert y\Vert_{2}^{2}+\sum_{i=1}^{k-1}\Vert
y^* P_{i}(\delta)\Vert_{2}^{2} = \|y\|_2^2 + \sum_{i=1}^{k-1} \|\delta^{i-k} y^* P^{k-i-1}(\delta) \|_2^2 \nonumber\\
& \leq \|y\|_2^2 + \sum_{i=1}^{k-1}|\delta|^{2(i-k)} \|P^{k-i-1}(\delta)\|_2^2 \|y\|_2^2 \nonumber\\
&\leq \|y\|_2^2 \max_{i=0:k}\{1, \|A_i\|_2^2\} \left[ 1 + \sum_{i=1}^{k-1} |\delta|^{2(i-k)}\left(\sum_{j=0}^{k-i-1}|\delta|^j\right)^2 \right]\nonumber\\
& \leq \|y\|_2^2 \max_{i=0:k}\{1, \|A_i\|_2^2\} \left[ 1 + \sum_{i=1}^{k-1} (k-i) \sum_{j=0}^{k-i-1}|\delta|^{2i-2k+2j}\right] \nonumber\\
& \leq \|y\|_2^2 \max_{i=0:k}\{1, \|A_i\|_2^2\} \left[ 1 + (k-1)^2 \sum_{j=0}^{k-2} |\delta|^{2j+2-2k}\right]\nonumber\\
& \leq \|y\|_2^2 \max_{i=0:k}\{1, \|A_i\|_2^2\} \left[ 1 + (k-1)^3 |\delta|^{-2}\right]. \label{wCbound}
\end{align}

Since $|\delta|+k\leq(|\delta|+1)k,$ taking into account
(\ref{f3}) and the upper bound for $\Vert w\Vert_{2}^{2}/\Vert y\Vert_{2}^{2}$
in (\ref{wCbound}), we have
\begin{align*}
\frac{\kappa_{C_{1}}(\delta)}{\kappa_{P}(\delta)} &  \leq\frac{\Vert
w\Vert_{2}}{\Vert y\Vert_{2}}\frac{\;\max_{i=0:k}\{1,\Vert A_{i}\Vert_{2}
\}}{\min\{\Vert A_{k}\Vert_{2},\Vert A_{0}\Vert_{2}\}
}\frac{(1+|\delta|)k\sqrt{k}|\delta|^{k-1}}{1+|\delta|^{k}}\\
&=\rho' k \sqrt{k} \sqrt{1+(k-1)^3|\delta|^{-2}} \frac{(1+|\delta|)}{|\delta|}\\
&\leq\left\{
\begin{array}
[c]{ll}%
2\sqrt{2} k^{3/2}\rho^{\prime}\leq \frac{4}{3} k(k+1)\rho', & \text{if $|\delta|\geq\sqrt{(k-1)^{3}},$}\\
 2\sqrt{2} k^{3}\rho^{\prime}, & \text{otherwise.}
\end{array}
\right.
\end{align*}

Now we find a lower bound for $\frac{\kappa_{C_1}(\delta)}{\kappa_P(\delta)}$. Notice that $\Vert w\Vert_{2}\geq\Vert y\Vert_{2}$. Thus, from (\ref{f1}) and (\ref{f2}), we have
\begin{align*}
\frac{\kappa_{C_{1}}(\delta)}{\kappa_{P}(\delta)} &\geq \frac{(|\delta| \|X\|_2 + \|Y\|_2) \|\Lambda(\delta)\|_2}{\sum_{i=0}^k |\delta|^i \|A_i\|_2}\\
& \geq \frac{min\{ \max\{1, \|A_k\|_2\}, \max_{i=0:k-1}\{1,\|A_i\|_2\}\} }{\max_{i=0:k} \{ \|A_i\|_2\}}\frac{(|\delta|+1)\sqrt{\sum_{i=0}^{k-1}|\delta|^{2i}}}{\sum_{i=0}^k |\delta|^i} 
\geq \frac{\nu}{k+1}
\end{align*}
where the last inequality holds since
$$\frac{(|\delta|+1)\sqrt{\sum_{i=0}^{k-1}|\delta|^{2i}}}{\sum_{i=0}^k |\delta|^i} \geq  \frac{1}{k+1}.$$
In fact, to see this, note that, if $|\delta| \geq 1$,
$$(|\delta|+1)\sqrt{ \sum_{i=0}^{k-1}|\delta|^{2i} }\geq |\delta|^k \quad \textrm{and} \quad \sum_{i=0}^k |\delta|^i \leq  (k+1) |\delta|^k,$$
and, if $|\delta| \leq 1$, 
$$(|\delta|+1)\sqrt{ \sum_{i=0}^{k-1}|\delta|^{2i} }\geq 1 \quad \textrm{and} \quad \sum_{i=0}^k |\delta|^i \leq k+1.$$
 \end{proof}
 
 
 
We now  recall a result obtained in \cite{backward}
regarding the comparison of the backward errors of  approximate eigenpairs of  $C_{1}(\lambda)$ and
$P(\lambda).$


 
 \begin{theorem}\cite[Theorems 3.6 and 3.8]{backward}\label{backwardC}
Let $P(\lambda)$ be a matrix polynomial of degree $k$ as in (\ref{pol}) with $A_0 \neq 0$.  Let $(z, \delta)$ be an approximate right eigenpair of $C_1(\lambda)$. Then, for $z_t=(e_t^T \otimes I_n)z$, $t=1:k$, we have that $(z_t, \delta)$ is an approximate right eigenpair of $P(\lambda)$ and 
 \begin{equation}\label{backC1}  \frac{\eta_P(z_t, \delta)}{\eta_{C_1}(z, \delta)} \leq k^{5/2} \frac{\|z\|_2}{\|z_t\|_2}\rho', 
 \end{equation}
 where $\rho'$ is as in (\ref{rhonu}). 
 Let $(w^*, \delta)$ be an approximate left eigenpair of $C_1(\lambda)$. Then, for $w_1 = (e_1^T \otimes I_n)w$, we have that $(w_1^*, \delta)$ is an approximate left eigenpair of $P(\lambda)$ and 
 $$\frac{\eta_P(w_1^*, \delta)}{\eta_{C_1}(w^*, \delta)} \leq k^{3/2} \frac{\|w\|_2}{\|w_1\|_2}\tau,$$
 where $\tau = \frac{\max_{i=0:k} \{1, \|A_i\|_2\}}{min \{ \|A_0\|_2, \|A_k\|_2\}}.$
 \end{theorem}

\begin{remark}\label{recovery2}
Taking into account the form of the right eigenvectors of $C_1(\lambda)$  (see Lemma \ref{eigC1}),  and assuming that the approximate eigenvector $z$ has a similar block structure, it is expected that the quotient $\frac{\|z\|_2}{\|z_t\|_2}$ in (\ref{backC1}) is close to 1  if $t=1$ and $|\delta| \geq 1$, or if $t=k$ and $|\delta| \leq 1$.  Thus, depending on the modulus of the approximate eigenvalue $\delta$, it is convenient to recover the approximate eigenvector $x$ of $P(\lambda)$ from the eigenvector $z$ of $C_1(\lambda)$ as follows: we take $x=z_1$ if $|\delta| \geq 1$ and $x=z_k$ if $|\delta| \leq 1$. 
Moreover, $\rho' \approx 1$ if  all matrix coefficients have norms close to 1. In practice, we can scale $P(\lambda)$ by dividing each of its matrix coefficients by $\max_{i=0:k}\{\|A_i\|_2\}$. The scaled polynomial $\tilde{P}(\lambda)$ has the same eigenvalues and eigenvectors as $P(\lambda)$ and, additionally, the conditioning of the eigenvalues and the backward error of the approximate eigenpairs is not changed by the scaling. Thus, we can work with the pencil $C_1(\lambda)$ associated with $\tilde{P}(\lambda)$. In this case, for $\rho'$ to be close to 1, it is enough that all matrix coefficients of $\tilde{P}(\lambda)$ have similar norms.
\end{remark}
 
 
 \begin{remark}\label{remarkC}
Based on Theorems \ref{conditioningR}, \ref{backwardR}, \ref{conditioningC} and \ref{backwardC}, we  conclude that $\mathcal{T}_p(\lambda)$ and $C_1(\lambda)$ have similar behavior in terms of conditioning and backward errors. Note that  the bounds on the quotients of condition numbers and backward errors are comparable when  $P(\lambda)$ is scaled by dividing its matrix coefficients by $\max_{i=0:k} \{\|A_i\|_2\}$. Moreover, both pencils are always  linearizations of $P(\lambda)$ (regular and singular). However, when $P(\lambda)$  is symmetric (Hermitian),  $\mathcal{T}_P(\lambda)$ has an advantage over $C_1(\lambda)$ since, in this case, $\mathcal{T}_P(\lambda)$ is also symmetric (Hermitian) while $C_1(\lambda)$ is not and, for numerical reasons, it is more convenient that the linearizations  preserve the structure of the original matrix polynomial in order to preserve any symmetries in the spectrum. 
 \end{remark}
 


\section{Numerical experiments}\label{numerical}

In this section, we run some numerical experiments to illustrate the theoretical results presented in previous sections concerning the conditioning of eigenvalues and backward errors of approximate eigenpairs of the linearizations $\mathcal{T}_P(\lambda)$, $D_1(\lambda,P)$, $D_k(\lambda,P)$ and $C_1(\lambda)$ of a matrix polynomial $P(\lambda)$ and to compare the behavior of the different linearizations analyzed. In some examples, $P(\lambda)$ is a random matrix polynomial while in others, $P(\lambda)$ is one of the matrix polynomials connected with some applications discussed in \cite{betcke2013nlevp}.


The experiments were  run on MATLAB-R2016a, for which the unit roundoff is $2^{-53}$. When calculating the condition numbers, we computed the eigenvalues, eigenvectors, and the condition numbers themselves  using variable precision arithmetic with 40 digits of precision. When computing the backward error, we considered only right eigenpairs. The function polyeig is used to compute approximate eigenvalues and eigenvectors of the linearization.  If $z$ denotes a computed eigenvector for the linearization associated with a computed eigenvalue $\delta$, an eigenvector for $P(\lambda)$ was recovered from $z$ as described in Remark \ref{recovery2} and Theorems  \ref{backwardR} and \ref{backwardD}.

In order to improve the ratio of the condition numbers and backward errors of the linearizations $\mathcal{T}_P(\lambda)$ and $C_1(\lambda)$ of the matrix polynomial $P(\lambda)$,  a scaling of the polynomial $P(\lambda)$ of the type discussed in Section \ref{cone-sec}  is applied in some cases (recall that the condition number and the backward error of $P(\lambda)$ are invariant under such  scalings). More specifically, in our experiments with random polynomials we scale $P(\lambda)$  by dividing each matrix coefficient by $\max_{i=0:k} \{\|A_i\|_2\}$, a scaling that leaves the parameter $\lambda$ invariant.
 As pointed out in Remarks \ref{remarkD1}, \ref{remarkD2}, and \ref{remarkC}, this scaling ensures that the parameters $\rho$, $\rho'$, $\rho_1$ and $\tau$ appearing in the bounds of the quotients of condition numbers and backward errors are of similar magnitudes and that the quotient of the norms of approximate eigenvectors in the upper bounds of the ratios of backward errors are close to 1. We note that this type of scaling does not produce any improvements in the ratios associated with $D_1(\lambda,P)$ and $D_k(\lambda,P)$. 
The numerical results that we present next show both cases, when $P(\lambda)$ is scaled and when $P(\lambda)$ is not.   In our experiments with the applied problems, we also apply some eigenvalue scalings, as we explain later.

To start with,  we examine a $20\times 20$  random matrix polynomial  $P(\lambda)$ of degree 3. The polynomial $P(\lambda)$ was generated by producing  random matrix coefficients with entries between $[-50,50]$.  Note that the norms of the matrix coefficients of $P(\lambda)$ are similar as  can be seen in Table \ref{Table1}. Also, in this table the smallest and the largest modulus of the eigenvalues are displayed.

\begin{figure}[H]
\includegraphics[width=16cm, height=9.6cm]{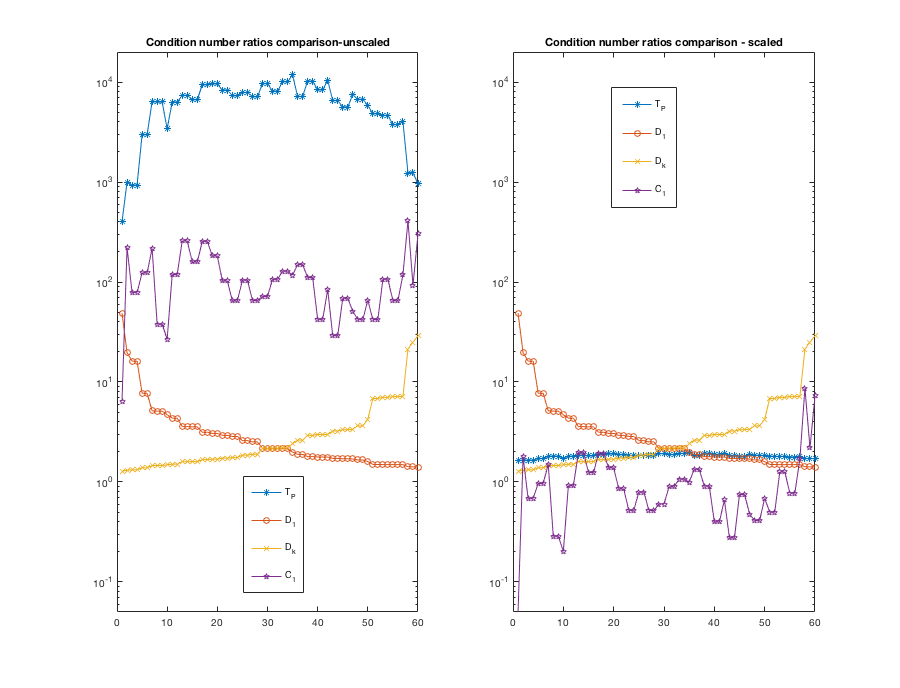}
\caption{Condition number ratio for unscaled (left) and scaled (right) random $20\times 20$ matrix polynomial $P(\lambda)$ of degree $3$.}\label{Figure1}\end{figure}

In Figures \ref{Figure1} and \ref{Figure2}, the $x$-axis has the indices $1,2,...,60$ to represent the non-zero, simple finite eigenvalues of $P(\lambda)$, which are sorted in increasing order by modulus (i.e. $1$ represents  the smallest eigenvalue in this order while 60 represents the eigenvalue with largest modulus). On the $y$-axis we give the ratio $\frac{\kappa_L(\delta)}{\kappa_P(\delta)}$, where $L(\lambda)$ denotes any of the  linearizations, $D_1(\lambda,P)$, $D_k(\lambda,P)$, $C_1(\lambda)$ or $\mathcal{T}_P(\lambda)$. 

In Figure \ref{Figure1}, we present the results for conditioning while in Figure \ref{Figure2} we present the results for backward errors. The conditioning  results associated with $\mathcal{T}_P(\lambda)$ in this example are not ideal when $P(\lambda)$ is not scaled (figure on the left), with ratios mostly close to $10^4$ compared to ratios close to 1 corresponding to the combined use of $D_1(\lambda,P)$ and $D_k(\lambda,P)$. Note that we can guess what eigenvalues have modulus less than 1 or larger than 1 by inspecting where the graphs for $D_1(\lambda,P)$ and $D_k(\lambda,P)$ intersect.  A  behavior similar to that of $\mathcal{T}_P(\lambda)$ can be observed for $C_1(\lambda)$, which is natural since, in contrast with $D_1(\lambda, P)$ and $D_k(\lambda, P)$,    $\mathcal{T}_p(\lambda)$ and $C_1(\lambda)$ have both identity blocks that will lead to undesirable behaviors either when $\max_{i=0:k} \{\|A_i\|_2\} \ll 1$ or when $\max_{i=0:k} \{ \|A_i\|_2\}\gg 1$. 
In this example, since 
$\max_{i=0:k} \{\|A_i\|_2\} \approx 250$, it can be expected that the behavior of $\mathcal{T}_p(\lambda)$ and $C_1(\lambda)$ is penalized by a  factor of  $\approx 10^4$ with respect to that of 
$D_1(\lambda, P)$ and $D_k(\lambda, P)$. We emphasize that these naturally expected behaviors are fully supported by the theoretical results in Section \ref{main} and \ref{cond-D1Dk}.

\begin{figure}[H]
\includegraphics[width=16cm, height=9.6cm]{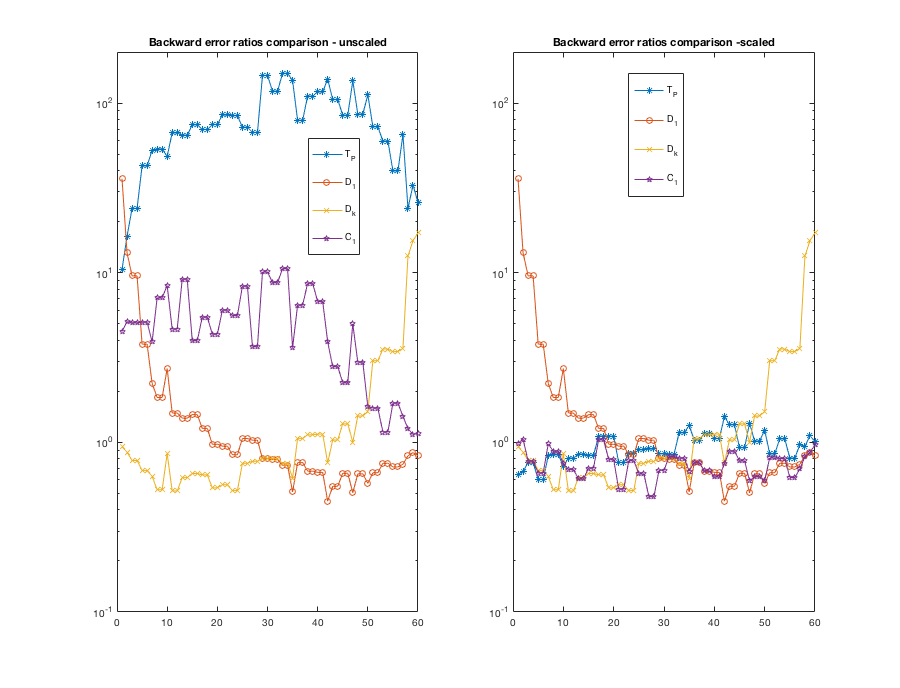}
\caption{Backward error ratio for unscaled (left) and scaled (right) random $20\times 20$ matrix polynomial $P(\lambda)$ of degree $3$.}\label{Figure2}\end{figure}





In the figure on the right in Figure \ref{Figure1}, we repeat the experiment with $P(\lambda)$ scaled by dividing all its matrix coefficients by $\max_{i=0:k}\{\|A_i\|_2\}$. Note that now  $\mathcal{T}_P(\lambda)$  yields remarkably better results, comparable to those of the combined used of $D_1(\lambda,P)$ and $D_k(\lambda,P)$ and much better than any of them considered individually.  It is worth noting that while the scaling improved the behavior of $\mathcal{T}_P(\lambda)$ and $C_1(\lambda)$, it had no effect on the condition number ratios for $D_1(\lambda,P)$ and $D_k(\lambda,P)$, as the theory predicts  and it is naturally expected as discussed in the previous paragraph. In Table \ref{Table1}, the value of the maximum ratio of condition numbers for each of the four linearizations has been highlighted in boldface.

In Figure \ref{Figure2} we present the values of the ratio $\frac{\eta_P(x, \delta)}{\eta_L(z, \delta)}$ of backward errors associated with $P(\lambda)$ and the linearizations that we are considering, before scaling (on the left) and after scaling (on the right). These figures show the improvement in the backward error ratio after scaling the same matrix polynomial  $P(\lambda)$ again by dividing all its matrix coefficients by $\max_{i=0:k}\{\|A_i\|_2\}$. Once more, we notice that the backward error ratios  for $\mathcal{T}_P(\lambda)$ and $C_1(\lambda)$ are worse than those of the combined use of $D_1(\lambda,P)$ and $D_k(\lambda,P)$  before scaling, but become much closer to 1 after scaling $P(\lambda)$. In Table \ref{Table1}, the value of the maximum ratio of backward errors for each of the four linearizations has been highlighted in boldface. Notice the good behavior of $\mathcal{T}_P(\lambda)$ compared to the other linearizations.  We can state, informally, that after scaling $\mathcal{T}_p(\lambda)$ is the winner among the block symmetric linearizations.

All data pertaining to Figures \ref{Figure1}-\ref{Figure2} is summarized in Table \ref{Table1}.  In the table we also add the data for a $2\times 2$ random matrix polynomial of degree 3. For this example we do not display the corresponding graphs, which are similar to the ones previously shown.  Note that the results do not seem to be influenced by the size of the matrix polynomial. 

One might expect that, for large $k$,  $\mathcal{T}_P(\lambda)$ has  a poor  behavior in terms of conditioning and backward error, especially for eigenvalues of modulus close to 1, since the upper bounds for the ratio of  the condition numbers and the ratios of  backward errors  increase with $k$ and, in fact, the bound for the ratio of condition
numbers increases faster than the one for $D_t(\lambda)$, $t=0,k$. However, optimal behavior is still observed for random matrix polynomials of large odd degree $k$ whose matrix coefficients have similar norms. Figure \ref{Figure5}  shows the results for a $3 \times 3$ random matrix polynomial  of degree $k = 21$ after scaling the matrix coefficients of the matrix polynomial by dividing by $\max_{i=0:k}\{\|A_i\|_2\}$.

\begin{table}
\caption{\footnotesize Experiment results for random polynomials: $k = 3, n = 2$ and $k = 3, n = 20$.}\label{Table1}
\begin{center}
\begin{tabular}{|c|cc|cc|}
\hline

Problem:  & \multicolumn{2}{|c|}{k=3} & \multicolumn{2}{|c|}{k=3}\\

 & \multicolumn{2}{|c|}{n=2}  & \multicolumn{2}{|c|}{n=20}\\
\hline
& Unscaled & Scaled  & Unscaled& Scaled \\
\hline
$|\delta_{min}|$ & 0.12&0.12&0.178 &0.178\\

$|\delta_{\max}|$ & 1.36 &1.36& 4.31 &4.31\\
\hline
$\|A_3\|_2$ & 41.7&0.75&228.07&0.94\\

$\|A_2\|_2$ &24.1&0.43&243.23&1\\

$\|A_1\|_2$ &45.2&0.81&235.7&0.97\\

$\|A_0\|_2$ &55.5&1&242.8&0.99\\
\hline
$\rho$ & 1.3 &1.3&1.066&1.066\\
$\rho_1$ &4117.02&1.3&63093.5&1.066\\
$\rho'$ &74.1&1.3&259.4&1.066\\
\hline & \space & \space & \space & \space\\
$D_1, D_k$ upper bound for cond &12 &12& 9.6  & 9.6 \\
$D_1$ upper bound for back &483.5 &483.5 &176.6&176.6 \\
$D_k$ upper bound for back &17.35 &17.35 &105.99 & 105.99\\
$\mathcal{T}_P$  upper bound for cond. &222319.3&72&3407051.27& 57.59\\
$\mathcal{T}_P$ upper bound for back. &55764.9&41.8&815295.5&33.84\\
$C_1$  upper bound for cond. &5658.99&101.86&19809.6& 81.44\\
$C_1$ upper bound for back. &1990.5&35.8&6845.6&28.1\\
& \space & \space & \space & \space\\
\hline
$\min\{\kappa_{D_1}/\kappa_{P}$\} &2.1&2.1&1.4&1.4\\
$\max\{\kappa_{D_1}/\kappa_{P}$\} &91.5&\textbf{91.5}&48.8&\textbf{48.8}\\
$\min\{\kappa_{D_k}/\kappa_{P}$\} &1.04 &1.04&1.28&1.28\\
$\max\{\kappa_{D_k}/\kappa_{P}$\} &3.4&\textbf{3.4}&28.96&\textbf{28.96}\\
$\min\{\kappa_{\mathcal{T}_P}/\kappa_{P} $\}&4.86&1.3&400.18&1.63\\
$\max\{\kappa_{\mathcal{T}_P}/\kappa_{P} $\}&1439.6&\textbf{2.65}&11812.15&\textbf{1.97}\\
$\min\{\kappa_{C_1}/\kappa_{P}$\} & 17.16&0.63&6.45&0.05\\
$\max\{\kappa_{C_1}/\kappa_{P}$\} &78.5&\textbf{22.38}&409.8&\textbf{8.69}\\
\hline
$\min\{\eta_P/\eta_{D_1} \}$& 0.27&0.27&0.45&0.45\\
$\max\{\eta_P/\eta_{D_1}\}$ &74.12&\textbf{74.12}&35.8&\textbf{35.8}\\
$\min\{\eta_P/\eta_{D_k}$\} & 0.43&0.43&0.52&0.52\\
$\max\{\eta_P/\eta_{D_k}$\}&0.99&\textbf{0.99}&17.14&\textbf{17.14}\\
$\min\{\eta_P/\eta_{\mathcal{T}_P}$\} &1.15 &0.64&10.46&0.6\\
$\max\{\eta_P/\eta_{\mathcal{T}_P}$\} &42.5&\textbf{1.74}&149.7&\textbf{1.41}\\
$\min\{\eta_P/\eta_{C_1}$\} &2.5&0.41&1.11&0.48\\
$\max\{\eta_P/\eta_{C_1}$\} &5.08&\textbf{1.22}&10.6&\textbf{1.05}\\

\hline
\end{tabular}
\end{center}
\end{table}

So far results for randomly generated matrix polynomials have been presented. Next we show the results of several  experiments on  matrix polynomials related with some applied problems given in \cite{betcke2013nlevp}.  We investigate the backward error ratios of two specific  problems: 1)  ``Relative Pose 5pt Problem" (five point relative pose problem in computer vision), which generates a $10\times 10$ matrix polynomial of degree $3$;   2)  ``Plasma-drift Problem" (modeling of drift instabilities in the plasma edge inside a Tokamak reactor), which generates a  $128\times 128$ matrix polynomial of degree $3$.  The results for condition number ratios are similar to the backward error results and, therefore, are omitted for brevity.

\begin{figure}[H]
\includegraphics[width=16cm, height=9.6cm]{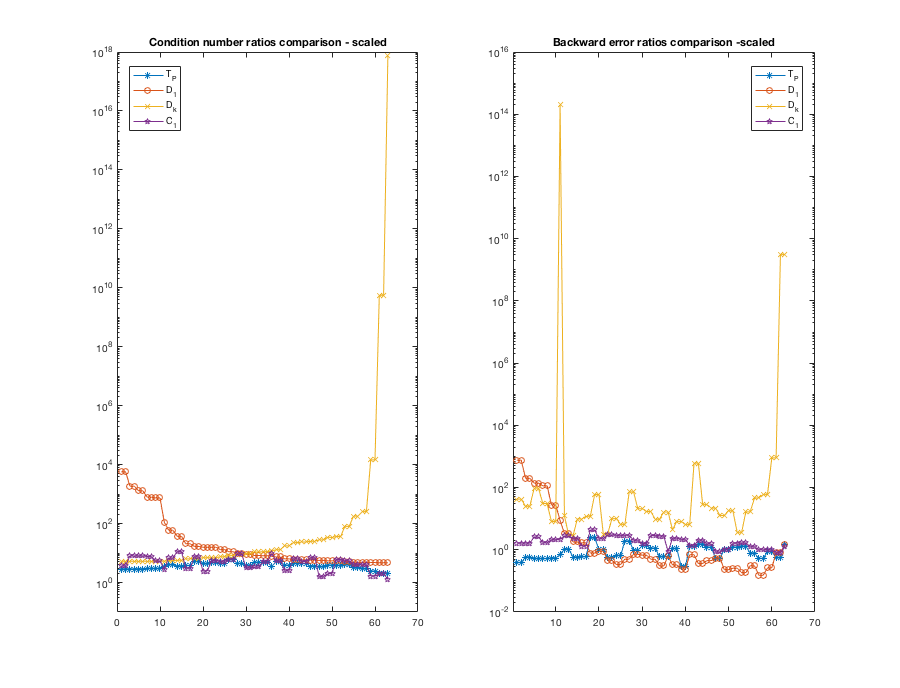}
\caption{Condition number ratio (left) and backward error ratio (right) for scaled random $3\times 3$   $P(\lambda)$ of degree $21$.}\label{Figure5}
\end{figure}



For the Relative Pose 5pt problem, since the polynomial has a singular leading coefficient, we only consider the linearizations $D_1(\lambda, P)$, $\mathcal{T}_P(\lambda)$, and $C_1(\lambda)$, since $D_k(\lambda, P)$ is not a linearization of the polynomial in this case. We note  that $P(\lambda)$ only has 10 finite eigenvalues. Moreover, the norms of the matrix coefficients of $P(\lambda)$ are similar. In this case the quotients of backward errors obtained associated with the three linearizations are close to 1 or less than 1, even when the polynomial is not scaled. Figure $\ref{Figure12}$ shows the results after scaling the polynomial by dividing its matrix coefficients by $\max_{i=0:k}\{\|A_i\|_2\}$. In Table \ref{Table3} we show data relative to this problem before and after scaling. 

The plasma-drift problem is more complex. As Table \ref{Table3} shows, the norms of the matrix coefficients of the matrix polynomial are not similar. Figure   \ref{Figure14} shows the ratios of backward errors corresponding to the unscaled polynomial (left figure) and the polynomial scaled by dividing its matrix coefficients by $\max_{i=0:k}\{\|A_i\|_2\}$ (right figure). This scaling is clearly  not enough  for $\mathcal{T}_{P}(\lambda)$ to outperform the combined use of $D_1(\lambda, P)$ and $D_k(\lambda, P)$ although the maximal backward error ratio for $\mathcal{T}_p(\lambda)$ is 35.28, which is very satisfactory. Thus, we consider two tropical scalings of the eigenvalue \cite{vanni, gaubert, tropical}. More explicitly, we consider the scaled polynomials $\beta P(\gamma \mu)$, where, for the tropical scaling 1, we use the parameters $\beta \approx 0.0081$ and $\gamma \approx 0.1$; and for the tropical scaling 2, we use the parameters $\beta\approx 0.0001$ and $\gamma \approx 9.8541$ (see Sections 2.1 and 2.2 in \cite{tropical} on how to calculate these parameters). The  obtained results are presented in Figure \ref{Figure15}. We note that, in this case, using the tropical scaling 1, $\mathcal{T}_P(\lambda)$ outperforms the combined use of $D_1(\lambda,P)$ and $D_k(\lambda, P)$ at one third of the eigenvalues (the ones with smallest modulus) but for large eigenvalues the behavior of $\mathcal{T}_p(\lambda)$ is not satisfactory and $D_1(\lambda, P)$ is clearly the winner; using the tropical scaling 2,  $\mathcal{T}_P(\lambda)$  produces ratios close to 1 for the rest of the eigenvalues. 

\begin{center}
\begin{figure}[H]
\includegraphics[width=15cm, height=9cm]{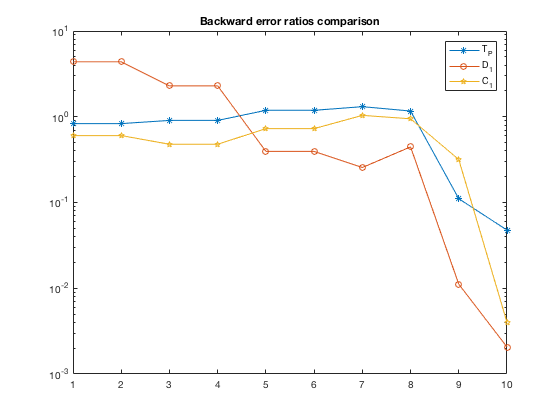}
\caption{Backward error ratios after scaling $P(\lambda)$ in Relative Pose 5pt Problem.}\label{Figure12}

\end{figure}
\end{center}

\begin{figure}[H]
\includegraphics[width=16cm, height=9.6cm]{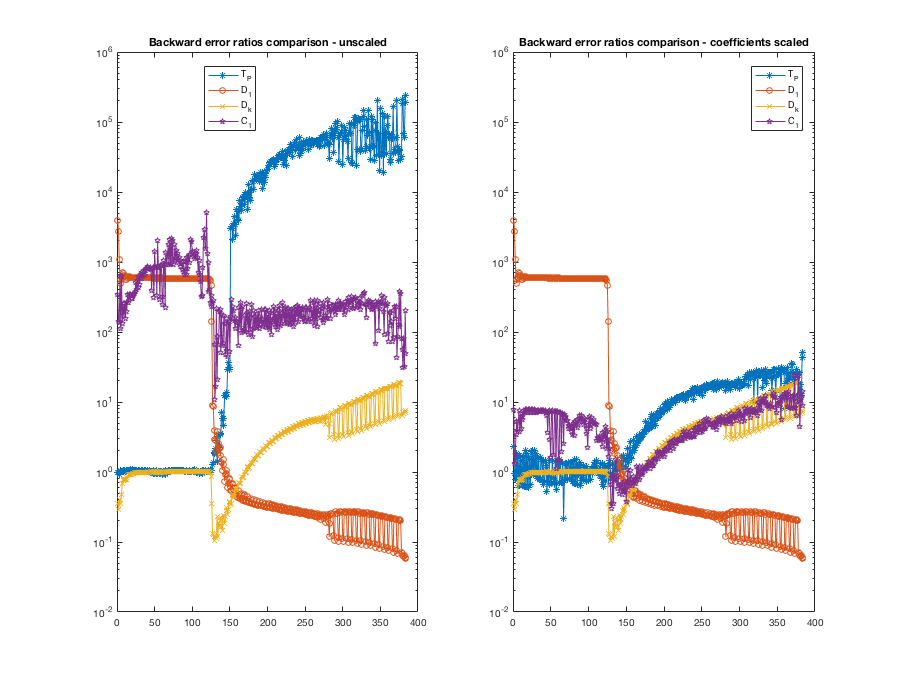}
\caption{Backward error ratios before scaling (left) and  scaling only coefficients (right)  in Plasma Drift Problem.}\label{Figure14}

\end{figure}

\begin{figure}[H]
\includegraphics[width=16cm, height=9.6cm]{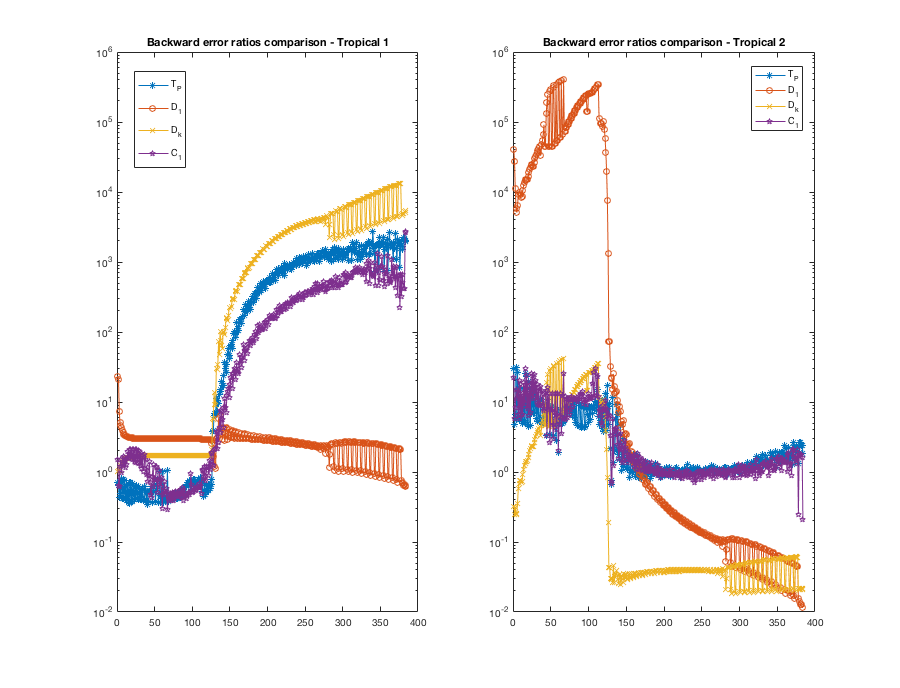}
\caption{Backward error ratios after Tropical scaling 1 (left) and  after Tropical scaling 2 (right)  in Plasma Drift Problem.}\label{Figure15}
\end{figure}








\begin{table}
\caption {\footnotesize Backward error results for two problems in NLEVP: A collection of nonlinear eigenvalue problems.}\label{Table3}
\begin{center}
\begin{tabular}{|c|cc|cccc|}
\hline
Problem:  & \multicolumn{2}{|c|}{Relative Pose 5pt} & \multicolumn{4}{|c|}{Plasma-drift}\\
\hline
& Unscaled & Scaled  & Unscaled& Scaled & Tropical 1 & Tropical 2 \\
\hline
$|\delta_{min}|$ &0.573&0.573&0.028&0.028 &0.28& 0.0028\\

$|\delta_{\max}|$ & 29.71 &29.71&11.745&11.745 &117.45& 1.1919\\

\hline
$\|A_3\|_2$ &0.708&0.336&12.698&0.0103 &0.0001& 1\\

$\|A_2\|_2$ &1.29&0.614&5.304&0.0043 &0.0004& 0.042\\

$\|A_1\|_2$ &1.88&0.891&1233.03&1.0 &1& 1\\

$\|A_0\|_2$ &2.11&1&123.23&0.0999 &1&0.01\\
\hline

$\min\{\eta_P/\eta_{D_1} \}$&0.002&0.002&0.0577&0.0577&0.61& 0.011\\
$\max\{\eta_P/\eta_{D_1}\}$ &4.372&4.372&3980.4&3980.4&23.52&406578.4\\
$\min\{\eta_P/\eta_{D_k} \}$&N/A&N/A&0.104&0.104&1.027&0.018\\
$\max\{\eta_P/\eta_{D_k}\}$ &N/A&N/A&18.749&18.749&13278.8&41.74\\
$\min\{\eta_P/\eta_{\mathcal{T}_P}$\} &0.033 &0.072&0.857&0.359&0.34&0.66\\
$\max\{\eta_P/\eta_{\mathcal{T}_P}$\} &1.18&1.45&190738.36&35.28&2764.58&30.59\\
$\min\{\eta_P/\eta_{C_1}$\} &0.112 &0.044&13.79&0.341&0.28&0.2\\
$\max\{\eta_P/\eta_{C_1}$\} &1.19&1.28&1883.16&14.24&2682.44&29.7\\
\hline
\end{tabular}
\end{center}
\end{table}

As a conclusion, in all the numerical experiments that we have run, we were able to find appropriate scalings of the matrix polynomial $P(\lambda)$ that allowed the use of $\mathcal{T}_P(\lambda)$ to produce small ratios of condition numbers for all nonzero finite simple eigenvalues of $P(\lambda)$, and small ratios of backward errors for approximate eigenpairs corresponding to all such eigenvalues. When the norms of the matrix coefficients of the matrix polynomial were similar, just scaling the matrix coefficients so that the largest one had norm 1 was enough for  the numerical behavior of $\mathcal{T}_P(\lambda)$ to be comparable with the combined used of $D_1(\lambda,P)$ and $D_k(\lambda,P)$, i.e.,   of the other two block symmetric linearizations considered in this work. When the norms of the matrix coefficients of $P(\lambda)$ were not similar, an eigenvalue tropical scaling was necessary.

\section{Conclusions}

In this paper, we have studied for the first time the eigenvalue conditioning and backward errors of  the linearization $\mathcal{T}_P(\lambda)$ (and $\mathcal{R}_P(\lambda)$) of an odd degree regular matrix polynomial $P(\lambda)$ and compared them with those of $D_1(\lambda,P)$, $D_k(\lambda,P)$, and $C_1(\lambda)$.  We  compared  the condition number of a finite, nonzero, simple eigenvalue $\delta$ of $P(\lambda)$, when considered an eigenvalue of $P(\lambda)$ and  an eigenvalue of each of the linearizations mentioned above. We have also compared the backward error of an approximate eigenpair $(z, \delta)$ of each of the previous  linearizations with the backward error of an approximate eigenpair $(x, \delta)$ of $P(\lambda)$, where $x$ has been recovered from $z$ in a convenient way.  As follows from the discussion below, when $P(\lambda)$ is symmetric or Hermitian and has odd degree, among the four linearizations studied, the linearization $\mathcal{T}_P(\lambda)$ (and $\mathcal{R}_P(\lambda)$) seems to be the most convenient to use in practice. 

For every regular matrix polynomial $P(\lambda)$ as in (\ref{pol}), the pencil $C_1(\lambda)$ is a strong linearization of $P(\lambda)$.  According to \cite{backward, tisseur} and our results in Section \ref{cond-D1Dk}, this linearization exhibits good numerical behavior, both in terms of conditioning and backward errors,  when $P(\lambda)$ has matrix coefficients with similar norm, specially after the matrix coefficients of $P(\lambda)$ are scaled by dividing each of them by $\max_{i=0:k} \{\|A_i\|_2\}$. The main disadvantage of this linearization is that it is not symmetric (Hermitian) when $P(\lambda)$ is, and preserving the structure of $P(\lambda)$ is convenient in order to preserve symmetries in the spectrum  in the presence of rounding errors (see, for instance, Section 3 in \cite{SC}). 

The pencils $D_1(\lambda,P)$ and $D_k(\lambda,P)$ are both symmetric (Hermitian) when $P(\lambda)$ is. However, $D_1(\lambda,P)$ and $D_k(\lambda,P)$ are not always  a linearization of a regular matrix polynomial $P(\lambda)$. It is well- known that $D_1(\lambda,P)$  is a linearization of $P(\lambda)$ if and only if $A_0$ is nonsingular and $D_k(\lambda,P)$ is a linearization of $P(\lambda)$ if and only if  $A_k$ is nonsingular. According to \cite{tisseur} and our results in Section \ref{cond-D1Dk}, when $D_1(\lambda,P)$ and $D_k(\lambda,P)$ are linearizations of $P(\lambda)$, they exhibit good numerical behavior in terms of  condition number but only in a certain range of eigenvalues,  and as long as the matrix coefficients of $P(\lambda)$ have similar norms. More precisely, the eigenvalue $\delta$ has  condition number ratio close to 1, when considered an eigenvalue of  $D_1(\lambda,P)$,  if $|\delta| \geq 1$, and when considered an eigenvalue of $D_k(\lambda,P)$, if $|\delta| \leq 1$. Similar conclusions can be drawn regarding the backward error,  as shown in \cite{backward}.

For every matrix polynomial $P(\lambda)$ of odd degree, the pencils $\mathcal{T}_P(\lambda)$ and $\mathcal{R}_P(\lambda)$  are strong linearizations of $P(\lambda)$. Moreover, they are symmetric (Hermitian) when $P(\lambda)$ is.  These pencils exhibit good numerical behavior in terms of conditioning and backward error, for eigenvalues of any modulus, as long as  the norms of the matrix coefficients are similar. Thus, we can conclude that, for regular symmetric or Hermitian  matrix polynomials of odd degree, $\mathcal{T}_P(\lambda)$ and $\mathcal{R}_P(\lambda)$ are the ideal choice among the linearizations that we have considered in this paper. It is an open question though  if we can find  linearizations for matrix polynomials of even degree  with  properties similar to those of $\mathcal{T}_P(\lambda)$ (and $\mathcal{R}_P(\lambda)$). 

\section{Appendix A}

Next we give the proof of  Theorem \ref{SC-preserv}.

Given a matrix polynomial $P(\lambda)$ as in (\ref{pol}), let $C_{P}$ and
$B_{P}$ be the matrices defined in \cite[Section 2]{SC}. The 
sign characteristic of a matrix polynomial is  the sign characteristic of the pair $(C_P, B_P)$ or 
any other pair unitarily similar to it (See Definition 2.3 in \cite{SC}).
We recall that unitarily similar pairs have the same
sign characetristic.

Let us denote 
$
\Gamma_l:=diag(1,\gamma,\gamma^{2},\ldots,\gamma^{l-1}).
$
It is easy to show that
\[
C_{\widetilde{P}}=\frac{1}{\gamma}\Gamma_k^{-1}C_{P}\Gamma_k,\quad B_{\widetilde
{P}}=\gamma\beta(\Gamma_k B_{p}\Gamma_k),
\]

Denote by $J_{P}$ and $J_{\widetilde{P}}$ the Jordan forms of $C_{P}$ and
$C_{\widetilde{P}}$, respectively, and let $S_{P}$ and $S_{\widetilde{P}}$ be
nonsingular matrices such that
\[
J_{P}=S_{P}^{-1}C_{P}S_{P},\quad J_{\widetilde{P}}=S_{\widetilde{P}}%
^{-1}C_{\widetilde{P}}S_{\widetilde{P}}.
\]
Let $N_{P}=S_{P}^{\ast}B_{P}S_{P}$ and $N_{\widetilde{P}}=S_{\widetilde{P}%
}^{\ast}B_{\widetilde{P}}S_{\widetilde{P}}.$ Note that  $(J_P, N_P)$ is unitarily similar to $(C_P, B_P)$, and $(J_{\tilde{P}}, N_{\tilde{P}})$ is unitarily similar to $(C_{\widetilde{P}}, B_{\widetilde{P}})$. This implies that the sign characteristic of $P$ and $\tilde{P}$ is given by the sign characteristic of $(J_P, N_P)$ and $(J_{\tilde{P}}, N_{\tilde{P}})$, respectively. We have
\begin{align*}
J_{\widetilde{P}} &  =\frac{1}{\gamma}(S_{\widetilde{P}}^{-1}\Gamma_k^{-1}%
)C_{P}(\Gamma_k S_{\widetilde{P}})
  =\frac{1}{\gamma}(S_{\widetilde{P}}^{-1}\Gamma_k^{-1}S_{P})J_{P}(S_{P}%
^{-1}\Gamma_k S_{\widetilde{P}})
  =(S_{P}^{-1}\Gamma_k S_{\widetilde{P}})^{-1}\left(  \frac{1}{\gamma}%
J_{P}\right)  (S_{P}^{-1}\Gamma_k S_{\widetilde{P}}).
\end{align*}
Also,
\begin{align*}
N_{\widetilde{P}} &  =\gamma\beta(S_{\widetilde{P}}^{\ast}\Gamma_k)B_{P}(\Gamma_k
S_{\widetilde{P}})
  =\gamma\beta(S_{\widetilde{P}}^{\ast}\Gamma_k S_{P}^{-\ast})N_{P}(S_{P}%
^{-1}\Gamma_k S_{\widetilde{P}})
  =(S_{P}^{-1}\Gamma_k S_{\widetilde{P}})^{\ast}(\gamma\beta N_{P})(S_{P}%
^{-1}\Gamma_k S_{\widetilde{P}}).
\end{align*}
Therefore, $(J_{\widetilde{P}},N_{\widetilde{P}})$ is unitarily similar
to $(\frac{1}{\gamma}J_{P},\gamma\beta N_{P})$. Thus, in order to prove our claim, it is enough to see that
$(\frac{1}{\gamma}J_{P},\gamma\beta N_{P})$ and $(J_{P},N_{P})$ have the same
sign characteristic.

Let $T$ be such that $T^{-1}J_{P}T=J$ and $T^{\ast}N_{P}T=P_{\varepsilon,J},$
where $J$ and $P_{\varepsilon,J}$ are as in Theorem 2.2 in \cite{SC}. Using the
notation for $\alpha$, $\beta$, and $l_i$, $i=1, \ldots, \beta$, in that theorem, let%
\[
L=
{\displaystyle\bigoplus_{i=1}^{\alpha}}
\sqrt{(\beta\gamma^{l_{i}})^{-1}}\Gamma_{l_i}\oplus
{\displaystyle\bigoplus_{i=\alpha+1}^{\beta}}
\left[  \sqrt{(\beta\gamma^{l_{i}})^{-1}}\left( \Gamma_{\frac{l_{i}}{2}}\oplus \Gamma_{\frac{l_{i}}{2}}\right)  \right] .
\]
Then,%
\[
(TL)^{-1}\left(\frac{1}{\gamma}J_{P}\right)TL=L^{-1}\left(\frac{1}{\gamma}J\right)L=J^{\prime},
\]
where $J^{\prime}$ has the same structure as $J,$ with each eigenvalue in $J$
multiplied by $\frac{1}{\gamma}$. Also,
\[
(TL)^{\ast}(\gamma\beta N_{P})TL=\gamma\beta L^{\ast}P_{\varepsilon
,J}L=P_{\varepsilon,J}.
\]
Thus, $(\frac{1}{\gamma}J_{P},\gamma\beta N_{P})$ and $(J_{P},N_{P})$ have the
same sign characteristic associated with the corresponding eigenvalues.


\end{document}